\documentclass[10pt,a4paper,oneside,titlepage]{article}
\usepackage[T1]{fontenc}
\usepackage[utf8]{inputenc}
\usepackage[english]{babel}

\usepackage[sorting=nyt]{biblatex}
\usepackage{csquotes}

\usepackage{amsfonts}
\usepackage{amssymb}
\usepackage{amsmath}
\usepackage{amsthm}
\usepackage{relsize}
\usepackage{bbm}
\usepackage{paralist}

\usepackage{bbm}
\usepackage{mathtools}
\usepackage{braket}
\usepackage{mathrsfs}
\usepackage{tensor}
\usepackage{cancel}
\usepackage{slashed}
\usepackage{xfrac}
\usepackage{faktor}

\usepackage{enumitem}
\usepackage{verbatim}
\usepackage{tikz}

\usepackage{graphicx}
\usepackage{float}
\usepackage{wrapfig}
\usepackage[]{geometry}

\usepackage{extpfeil}

\usepackage{caption}

\usetikzlibrary{arrows,arrows.meta,petri,topaths,positioning,shapes,shapes.misc,patterns,calc,decorations.pathreplacing,hobby}

\usepackage[backgroundcolor=blue!5!white,textsize=scriptsize]{todonotes}

\usepackage{hyperref}

\theoremstyle{plain}
\newtheorem{Thm}{Theorem}[section]

\theoremstyle{definition}
\newtheorem{Def}[Thm]{Definition}

\theoremstyle{plain}
\newtheorem{Lem}[Thm]{Lemma}

\theoremstyle{plain}
\newtheorem{Prop}[Thm]{Proposition}

\theoremstyle{plain}
\newtheorem{Cor}[Thm]{Corollary}

\theoremstyle{remark}
\newtheorem{Rmk}[Thm]{Remark}

\theoremstyle{remark}
\newtheorem*{Ex}{Example}

\theoremstyle{definition}

\DeclareFontFamily{U}{mathc}{}
\DeclareFontShape{U}{mathc}{m}{it}%
{<->s*[1.03] mathc10}{}
\DeclareMathAlphabet{\mathscr}{U}{mathc}{m}{it}

\newcommand{\F}{\mathcal{F}}
\newcommand{\K}{\mathsf{K}}
\newcommand{\Hom}{\mathrm{Hom}}
\renewcommand{\hom}{\mathrm{hom}}
\newcommand{\Th}{\mathcal{T}}
\renewcommand{\S}{\mathcal{S}}
\newcommand{\Coalg}{\mathsf{Coalg_{coc}}}
\renewcommand{\Set}{\mathsf{Set}}
\newcommand{\Hopfc}{\mathsf{Hopf_{coc}}}
\newcommand{\HBr}{\mathsf{HBr_{coc}}}
\newcommand{\Hopfcc}{\mathsf{Hopf_{coc}^{com}}}
\newcommand{\N}{\mathbb{N}}

\newcommand{\ra}{\rightarrow}
\newcommand{\T}[1]{\mathcal{T}(#1)}
\newcommand{\C}{\mathbf{C}}
\def\colim{\qopname\relax m{colim}}

\newcommand{\id}{\mathrm{id}}
\newcommand{\Grp}{\mathsf{Grp}}
\newcommand{\Ab}{\mathsf{Ab}}
\newcommand{\Bialg}{\mathsf{Bialg_{coc}}}

\newcommand{\SKB}{\mathsf{SKB}}
\newcommand{\RadRng}{\mathsf{RadRng}}
\newcommand{\DiGrp}{\mathsf{DiGrp}}
\newcommand{\HRadRng}{\mathsf{HRadRng_{coc}}}
\newcommand{\HDiGrp}{\mathsf{HDiGrp_{coc}}}

\newextarrow{\xbigtoto}{{40}{40}{20}{20}}
   {\bigRelbar\bigRelbar{\bigtwoarrowsleft\rightarrow\rightarrow}}

\DeclareMathOperator{\Hker}{Hker}
  
\usetikzlibrary{shapes.geometric, arrows.meta}

\NewBibliographyString{toappear}
\DefineBibliographyStrings{english}{%
  toappear = {to appear},
}

\renewbibmacro*{in:}{%
  \iffieldundef{pubstate}
    {}
    {\printfield{pubstate}%
     \setunit{\addspace}%
     \clearfield{pubstate}}%
  \bibstring{in}%
  \printunit{\intitlepunct}}

\begin{document}
\begin{center}
{\LARGE Coalgebraic models of $\Omega$-groups}

\vspace{0.5cm}
{\large Maria Bevilacqua}

\vspace{0.3cm}
{\large \textsc{Université catholique de Louvain}}

\vspace{0cm}
{ Institut de Recherche en Mathématique et Physique}

\vspace{0cm}
{\small \textit{Chemin du Cyclotron 2, Louvain-la-Neuve, 1348, Belgium}}

\vspace{0.2cm}
\texttt{maria.bevilacqua@uclouvain.be}
\end{center}

{\begin{center}
 \large\textsc{Abstract}
 \vspace{0.3cm}

 \begin{minipage}{0.7\textwidth}
    \small{We investigate models of algebraic theories in the category of cocommutative coalgebras over a field. 
    We establish some of their categorical properties, similar to those of algebraic varieties. We introduce a class of categories of coalgebraic models of algebraic theories endowed with an underlying structure of cocommutative Hopf algebra, and show that these categories are semi-abelian. We call them ``categories of $\Omega$-Hopf algebras'', since it is possible to characterize them as coalgebraic models of algebraic theories of  $\Omega$-groups in the sense of Higgins.
    
\vspace{0.2cm}

}
\end{minipage}

\end{center}}


\begin{section}{Introduction}

In the literature, several well-studied algebraic structures are endowed with an underlying structure of coalgebra over a field. This is the case for Hopf algebras---perhaps the most well-known example---but also for bialgebras and Hopf braces \cite{AngionoGalindoVendramin}. When the coalgebra structure is cocommutative, these objects can be regarded as product-preserving functors from a Lawvere theory to the category of cocommutative coalgebras, namely as coalgebraic models of a suitable algebraic theory. 

A first goal of this paper is to initiate a systematic investigation of coalgebraic models of algebraic theories, in order to develop a unified framework in which the aforementioned structures fit. It is possible to define algebraic functors between categories of coalgebraic and---as in the classical setting---we can think of them as functors forgetting some structure (axioms or operations).
After establishing some categorical and algebraic properties of coalgebraic models and algebraic functors between them, we introduce a class of categories of coalgebraic models with remarkable regularity and exactness properties. These categories arise as categories endowed with a forgetful functor to the category $\Hopfc$ of cocommutative Hopf algebras. Nevertheless, a motivating point in studying them is that they can be regarded as the coalgebraic models for theories of $\Omega$-groups. Algebraic varieties of $\Omega$-groups were introduced by Higgins \cite{Higgins} as a general setting for studying rings, Lie algebras, Jordan algebras, and other group-like structures. Indeed, an $\Omega$-group is a set equipped with a group structure together with additional operations that satisfy the following condition: every operation, when evaluated at the zero element, yields the zero element. In Proposition \ref{Characterization of Omega-groups}, we prove that being a category of models for these algebraic theories, in the coalgebraic context, translates in having a zero object and an underlying structure of cocommutative Hopf algebra. The main theorem we proved in Section \ref{omega groups}---Theorem \ref{semiabelianess of omega hopf algebras}---shows that these categories inherit the property of being semi-abelian from $\Hopfc$.  

Semi-abelian categories were introduced in \cite{JanelidzeMarkiTholen} as a general framework for the study of homology in non-abelian categories. The aim was to capture some properties of the category of groups and Lie algebras, similarly to
the way the notion of abelian category captures the properties of abelian
groups and modules over a ring. Indeed, semi-abelian categories can be regarded as a weaker version of abelian categories, in which the additivity does not hold in general. More precisely, a semi-abelian category is a pointed category that is Barr-exact, protomodular \cite{Bourn1991Normalization} (that is equivalent, in this context, to the validity of the Split Short Five Lemma) and has binary coproducts. Groups, loops, Lie algebras, cocommutative Hopf algebras, or crossed modules are some important examples of semi-abelian categories. 

As we mentioned before, in this manuscript, we provide a new class of examples of semi-abelian categories. Since they arise as coalgebraic models for Lawvere theories of $\Omega$-groups, we dedicate the first part of the article to describing the general process of passing from set-theoretic models of algebraic theories to coalgebraic ones. 
We motivate this process by mentioning a lot of examples of categories that can be regarded as categories of coalgebraic models for suitable algebraic theories. Table \ref{Table of examples} summarizes all the examples compared with the respective models in $\Set$.

As in the case of algebraic varieties, an important role is played by terms, which are, in this coalgebraic setting, homomorphisms of coalgebras. Following \cite{Izquierdo}, we call them ``linearized terms''. It turns out that linearized terms encode structural properties of the corresponding category of coalgebraic models: for instance, protomodularity can be characterized by means of linearized terms, in the same spirit as the characterization of protomodular varieties of universal algebra \cite{borceuxbourn}. We show this result in Corollary \ref{linearized protomod} and in Proposition \ref{Protomodularity mal'tsev condition}.

Further analogies between classical and coalgebraic models also arise.  In Section \ref{section free functor}, we describe the construction of left adjoints of algebraic functors and, in Section \ref{Section Colimts}, we provide an explicit construction of colimits in a category of coalgebraic models. 

Next, in Section \ref{section full functors}, we study algebraic functors induced by surjective morphisms of theories---it is the case, for instance, of the algebraic functor from $\Hopfc$ to the category of commutative and cocommutative Hopf algebras. Left adjoints to these algebraic functors turn out to have an easier description as suitable quotients (we also apply this construction in some concrete examples). Furthermore, in analogy with the classical case, an algebraic functor induced by surjective maps of theories determines a Birkhoff subvariety, as shown in Proposition \ref{surj determines Birkhoff subvariety}.  

The final section builds on the well-known fact that, under the cocommutativity assumption, the category $\Hopfc$ of Hopf algebras enjoys many exactness properties and, in particular, it is semi-abelian \cite{GranSterckVercruysse}. The same holds for the category $\HBr$ of cocommutative Hopf braces, as recently proved in \cite{GranSciandra}. Interestingly, we notice that the crucial ingredient needed in the proof of semi-abelianness of $\HBr$ is the existence of a forgetful functor to $\Hopfc$. Motivated by this observation, we generalize this result by introducing a new class of semi-abelian categories characterized by the presence of an underlying cocommutative Hopf algebra structure. Since these categories correspond precisely to coalgebraic models of theories of $\Omega$-groups, we call them ``categories of $\Omega$-Hopf algebras''. Theorem \ref{semiabelianess of omega hopf algebras} leads to new examples of semi-abelian categories: for instance, the category of ``Hopf digroups'' (introduced in \cite{agore2025categoryhopfbraces}) and the category of ``Hopf radical rings''. Some of the aforementioned results are applied to investigate them.
    
\end{section}

\begin{section}{Linearizing process}\label{section
linearization}

In this section, we introduce a unified framework in which some categories whose objects have the property of having an underlying coalgebra structure can be studied. The idea comes from the third section of \cite{Izquierdo}, where the author described what he called ``A linearizing process". Given a type of algebras (in the universal algebra sense) P\'erez-Izquierdo associates, with a classical term in an algebraic theory, a linearized one that is a homomorphism of coalgebras. This process describes, for instance, how identities of groups move to identities of Hopf algebras, or how identities of loops move to those of non-associative Hopf algebras---the latter being the case examined in \cite{Izquierdo}. 

We briefly recall this construction and explain why it leads us to consider categories of product-preserving functors from a Lawvere theory to the category of cocommutative and coassociative coalgebras over a field $\K$.

A cocommutative coalgebra over $\K$ is a triple $(C,\Delta,\epsilon)$, where $C$ is a $\K$-vector space, $\Delta$ and $\epsilon$ are linear maps $\Delta\colon C\to C\otimes C$ and $\epsilon\colon C\to \K$---called comultiplication and counit---satisfying:
\begin{align*}
    (\Delta\otimes \id)\circ\Delta=&(\id\otimes\Delta)\circ\Delta,\\
    (\epsilon\otimes \id)\circ\Delta=\id&=(\id\otimes\epsilon)\circ\Delta,\\
    \tau\circ \Delta=&\Delta,
\end{align*}
where $\tau$ is the map that twists the two components of the tensor product.
A homomorphism of coalgebras $f\colon (A,\Delta^A,\epsilon^A)\ra (B,\Delta^B,\epsilon^B)$ is a linear map $f\colon A\to B$ such that $(f\otimes f)\circ \Delta^A=\Delta^B\circ f$ and $\epsilon^B\circ f=\epsilon^A$. We denote by $\Coalg$ the category of cocommutative coalgebras over $\K$ with homomorphisms of coalgebras between them. The category $\Coalg$ is a braided monoidal category, with braiding $\tau$.
Several properties of this category have been studied in \cite{GrunenfelderPare}.

We refer to \cite{BurrisSankappanavar1981} for some general notions of universal algebra, in particular those of algebras, operations, and terms.

Let $\F$ be a \textit{type of algebras}, namely a set of function symbols together with an assignment of a non-negative integer for each function symbol of $\F$, called \textit{arity}. We call $\F$-\textit{algebra} a set $A$ with a family of finitary operations on $A$ indexed on $\F$, such that, corresponding to each function symbol $f$ in $\F$, there is an $n$-ary operation $f^A$ on $A$, where $n$ is the arity of $f$.

\begin{Def}
    An $\F$-coalgebra is a cocommutative coalgebra $(C,\Delta,\epsilon)$ in $\Coalg$ that is also an $\F$-algebra such that every operation $f$ of arity $n$ induces a homomorphism of coalgebras from the $n$-th power of $C$ to $C$, that we denote as $f^C\colon C^{\otimes n}\rightarrow C$.
\end{Def}

\begin{Rmk}
    If $C$ is an $\F$-coalgebra and $V$ a coassociative and cocommutative coalgebra with comultiplication $\Delta$, then the vector space $\Hom(V,C)$ is an $\F$-algebra with operations defined by:
    $$f^{\Hom(V,C)}(\alpha_1,\dots,\alpha_k)\coloneqq f^C\circ(\alpha_1\otimes \dots\otimes\alpha_k)\circ \Delta^{(k-1)}$$
    for every $f$ of arity $k$ and every $\alpha_1,\dots,\alpha_k\in\Hom(V,C)$, where $\Delta^{(k-1)}$ denotes the $k-1$-iteration of the comultiplication.
    
\end{Rmk}

Let $X=\{x_1,\dots,x_n\}$ be a set of variables and let $T(X)$ be the term algebra. We define the homomorphism of algebras:
\begin{align*}
    l_n\colon T(X)&\longrightarrow \Hom(C^{\otimes n},C)\\
    x_i&\longmapsto \epsilon\otimes\dots\otimes \mathrm{id}\otimes\dots\otimes \epsilon,
\end{align*}
where the identity morphism $\mathrm{id}\colon C\to C$ lies in the $i$-th place.
If $s$ and $t$ are two terms in $T(X)$, we call the maps $l_n(s)$ and $l_n(t)$ \textit{linearized terms} and we say that an $\F$-coalgebra \textit{satisfies the equation} $s\approx t$ if and only if $l_n(s)=l_n(t)$.

\begin{Ex}
    If $\F$ is the type with an operation symbol $\cdot$ of arity $2$, the image of the term $x\cdot x$ is the map sending an element $c$ to the element $c_1\cdot^{C}c_2$, where we are using the Sweedler's notation and $\cdot^{C}$ is the operation interpreted in $C$ (we will omit the superscript when the context makes the interpretation clear).
\end{Ex}


A key result proved in \cite[Theorem 13]{Izquierdo} is the following one.

\begin{Thm}{\label{Izq}}
    Let $\Sigma$ be a set of equations and $s\approx t$ a consequence of $\Sigma$. If an $\F$-coalgebra $C$ satisfies all identities in $\Sigma$, then $l_n(s)=l_n(t)$, i.e. it satisfies $s\approx t$.
\end{Thm}

Let $\Sigma $ be an equational theory and let $\Th$ be the corresponding Lawvere theory \cite{LAWVERE2014413}(or algebraic theory), whose objects are the powers of $T$. There is a one-to-one correspondence between
$\F$-coalgebras satisfying all equations in $\Sigma$ and functors from $\Th$ to $\Coalg$ that preserve finite products. Any $\F$-coalgebra $(C,\Delta,\epsilon)$ is sent to the functor
\begin{align*}
    \mathbf{C}\colon \Th&\longrightarrow \Coalg\\
    T^n&\longmapsto C^{\otimes n}\\
    (f\colon T^n\rightarrow T)&\longmapsto (f^C\colon C^{\otimes n}\rightarrow C)
\end{align*}
which is well-defined by Theorem \ref{Izq}. Notice that---in the definition of this functor---we could not drop the cocommutativity assumption about the category of coalgebras, since it ensures that the categorical product is actually the tensor product, namely that the category is cartesian. Conversely, given a functor $\mathbf{C}\colon \Th\to \Coalg$, we can send it to $\mathbf{C}(T)$. It is straightforward to see that this correspondence is functorial and yields an isomorphism of categories.

\vspace{0.3cm}
Let us denote by $[\Th,\Set]$ the category of product-preserving functors from $\Th$ to $\Set$, i.e. the category of set-theoretic \textit{models of} $\Th$ and by $[\Th,\Coalg]$ the category of product-preserving functors from $\Th$ to $\Coalg$, that is the category of \textit{coalgebraic models of the Lawvere theory} $\Th$. Furthermore, we call an object in $[\Th,\Coalg]$ a \emph{$\Th$-coalgebra}.

\begin{subsection}{Examples}\label{Examples}
    Several well-known categories fit into the framework just introduced, namely can be regarded as categories of coalgebraic models of suitable algebraic theories.

    Our prototypal example is the category $\Hopfc$ of cocommutative Hopf algebras over a field, whose properties have been widely studied. Indeed, a cocommutative Hopf algebra can be viewed as a product-preserving functor from the algebraic theory $\Th_{\Grp}$ of groups to $\Coalg$. More explicitly, a group can be regarded as a set $G$ with functions 
    $$\cdot\;\colon G\times G\ra G\quad\quad 1\colon \{\star\}\ra G\quad\quad (-)^{-1}\colon G\ra G$$
    satisfying the equations:
    \begin{align*}
        x\cdot(y\cdot z)=&(x\cdot y)\cdot z,\\
        x\cdot 1=x&=1\cdot x,\\
        x\cdot x^{-1}=1&=x^{-1}\cdot x.        
    \end{align*}
    In the same way, a cocommutative Hopf algebra is a cocommutative coalgebra $(H,\Delta,\epsilon)$ together with coalgebra homomorphisms 
    $$m\colon H\otimes H\ra H,\quad\quad u\colon \K\ra H,\quad \quad S\colon H\ra H$$
    satisfying the identities:
    \begin{align*}
        m(a\otimes m( b\otimes c))=& m(m(a\otimes b)\otimes c),\\
        m(u(1_\K)\otimes a)=a&=m(a\otimes u(1_\K)),\\
        m(a_1\otimes S(a_2))=u(\epsilon&(a))=m(S(a_1)\otimes a_2),
    \end{align*}
    which are the linearized version of the previous equations.

    Similarly, the category $\Bialg$ of cocommutative bialgebras arises as the category of coalgebraic models of the Lawvere theory of monoids, while the category $\Hopfcc$ of commutative and cocommutative Hopf algebras can be regarded as the category of coalgebraic models of the Lawvere theory of abelian groups. Other examples from the non-associative setting include $H$-bialgebras and Moufang $H$-bialgebras (\cite{Izquierdo,Zhelyabin}), which are, respectively, the coalgebraic models of the theories of loops and Moufang loops. We denote these two categories respectively as $\mathsf{Loops}$ and $\mathsf{MLoops}$ and the corresponding categories of coalgebraic models as $\mathsf{HBialg_{coc}}$ and $\mathsf{MHBialg_{coc}}$.

The next example that we provide is given by the category of cocommutative Hopf braces that can be seen as the category of coalgebraic models of the algebraic theory of skew braces. Skew braces were introduced in \cite{GuarnieriVendramin} to study solutions of the set-theoretic Yang--Baxter equation, while Hopf braces have been introduced in \cite{AngionoGalindoVendramin} as a Hopf-theoretic generalization of skew braces, providing solutions for the quantum Yang--Baxter equation in the cocommutative setting. We recall the definitions below.

    \begin{Def}
        A \textit{skew (left) brace } is a triple $(A,+,\cdot)$ where $(A,+)$ and $(A,\cdot)$ are groups such that the following compatibility condition:
        $$a\cdot(b+c)=a\cdot b-a+a\cdot c$$
        holds for all $a,b,c$ in $A$.
        
    \end{Def}

    Notice that we use the additive notation for one of the two group operations, even though the groups involved in the definition are not assumed to be commutative.

    \begin{Def}
    A (cocommutative) \textit{Hopf brace} is the datum of $(H,\cdot,\bullet,1,\Delta,\epsilon,S,T)$ where $(H,\cdot,1,\Delta,\epsilon,S)$ and $(H,\bullet,1,\Delta,\epsilon,T)$ are (cocommutative) Hopf algebras satisfying:
    $$a\bullet(b\cdot c)=(a_1\bullet b)\cdot S(a_2)\cdot (a_3\bullet c).$$
    \end{Def}

    We denote by $\SKB$ the category of skew braces (where arrows are functions that are homomorphisms with respect to both the two group structures) and by $\Th_{\SKB}$ its corresponding Lawvere theory. Moreover, we denote by $\HBr$ the category of cocommutative Hopf braces.

    Since the compatibility axiom in the definition of a cocommutative Hopf brace is the linearization of the compatibility condition for skew brace, it is immediate to notice that $\HBr\cong[\Th_\SKB,\Coalg]$.

    Beyond these already known categories, this linearizing process paves the way for defining new, potentially interesting categories by taking models of arbitrary Lawvere theories in coalgebras instead of in $\Set$.
    In this spirit, we now define two categories (closely related to Hopf braces) whose properties will be analyzed in the following sections. In order to introduce them, we recall the definitions of some algebraic varieties.

    \begin{Def}[\cite{Bourn2000}]
        A \textit{digroup} is a triple $(A,+,\cdot)$ where $(A,+)$ and $(A,\cdot)$ are groups sharing the same neutral element. We denote by $\DiGrp$ the category of digroups, where the maps are homomorphisms of digroups, i.e. functions that are homomorphism of groups with respect to both structures. Furthermore, we denote by $\Th_\DiGrp$ the Lawvere theory of this algebraic variety.
    \end{Def}

    \begin{Def}[\cite{Rump}]
        A \textit{radical ring} is a skew brace $(A,+,\cdot)$ satisfying the identities:
        \begin{align*}
            (a+b)\cdot c=&a\cdot c-c+ b\cdot c\\
            a+b=&b+a.
        \end{align*}
        We denote by $\RadRng$ the category of radical rings (with homomorphisms of digroups) and by $\Th_\RadRng$ its corresponding Lawvere theory.
    \end{Def}
    We mention that this is a convenient reformulation of the identities of a radical rings for the purpose of this work. In particular, the equivalence between the classical definition and this one follows from Corollary of \cite[Proposition 4]{Rump}, using the fact that a brace---as defined in \cite{Rump}---is equivalent to a skew brace $(A,+,\cdot)$ such that the group $(A,+)$ is abelian. We also refer to \cite[Theorem 8.1.20]{Vendraminbook} for further details.

    Looking at algebraic varieties, there is a chain of functors of Lawvere theories:
    $$\Th_\DiGrp\xrightarrow{\hspace{0.8cm}}\Th_\SKB\xrightarrow{\hspace{0.8cm}}\Th_\RadRng.$$
    From the previous chain, one gets the following chain of algebraic functors 
    \begin{equation}\label{chain of inclusions}
        \RadRng\xrightarrow{\hspace{0.5cm}}\SKB\xrightarrow{\hspace{0.5cm}}\DiGrp,
    \end{equation}
    where the two arrows are inclusions of Birkhoff varieties whose properties have been studied in \cite{GranLetourmyVendramin}. In light of the good categorical properties of $\HBr\cong[\Th_\SKB,\Coalg]$---e.g. its semi-abelianness proved in \cite{GranSciandra}---it is also interesting to analyze the linearized version of $\RadRng$ and of $\DiGrp$. Hence, we define the category $\HRadRng$ of \textit{Hopf radical rings}, that is precisely the category $[\Th_\RadRng,\Coalg]$ and the category $\HDiGrp$ of \textit{Hopf digroups} that is $[\Th_\DiGrp,\Coalg]$.

    More explicit definitions of these categories can be formulated in the following way.

    \begin{Def}[\cite{agore2025categoryhopfbraces}]
       A \textit{(cocommutative) Hopf digroup} is the datum of $(H,\cdot,\bullet,1,\Delta,\epsilon,S,T)$ where $(H,\cdot,1,\Delta,\epsilon,S)$ and $(H,\bullet,1,\Delta,\epsilon,T)$ are (cocommutative) Hopf algebras.
       We denote by $\HDiGrp$ the category of Hopf digroups.
    \end{Def}

    Recall that a Hopf brace is a Hopf digroup $(H,\cdot,\bullet,1,\Delta,\epsilon,S,T)$ satisfying the compatibility condition:  
    $$a\bullet(b\cdot c)=(a_1\bullet b)\cdot S(a_2)\cdot (a_3\bullet c).$$
    
    \begin{Def}\label{Def HRadRng}
       A \textit{(cocommutative) Hopf radical ring} is a Hopf brace $(H,\cdot,\bullet,1,\Delta,\epsilon,S,T)$ satisfying the linearized equations:
       \begin{align*}
           (a\cdot b)\bullet c=& (a\bullet c_1)\cdot S(c_2)\cdot (b\bullet c_3)\\
           a\cdot b=& b\cdot a.
       \end{align*}
       We denote by $\HRadRng$ the category of Hopf radical rings.
    \end{Def}

   With these definitions, we obtain the linearized analogue of (\ref{chain of inclusions}), namely the following chain of functors between categories of coalgebraic models:
   \begin{equation}\label{Chain of Hopf inclusions}
       \HRadRng\xrightarrow{\hspace{0.5cm}}\HBr\xrightarrow{\hspace{0.5cm}}\HDiGrp
   \end{equation}

   Part of the aim of this paper is to study these inclusions. We will show that the two functors are inclusions of Birkhoff subcategories---as in (\ref{chain of inclusions})---and that all the categories involved are semi-abelian. 

   We summarize all the above-mentioned examples of categories of coalgebraic models of algebraic theories in the following table.

   \begin{table}[H]
   \caption{Examples}
   \centering
   \begingroup
   \setlength{\tabcolsep}{10pt}
   \renewcommand{\arraystretch}{1.3} 
   \label{Table of examples}
   \begin{tabular}{|c|c|c|}
   \hline
       \textbf{Lawvere theory} & \textbf{Models in $\Set$} & \textbf{Models in $\Coalg$}\\
       \hline
       $\Th_\mathsf{Mon}$ & $\mathsf{Mon}$ & $\Bialg$\\
       $\Th_\mathsf{Loops}$ & $\mathsf{Loops}$ & $\mathsf{HBialg_{coc}}$\\
       $\Th_\mathsf{MLoops}$ & $\mathsf{MLoops}$ & $\mathsf{MHBialg_{coc}}$\\
       $\Th_\Grp$ & $\Grp$ & $\Hopfc$\\
       $\Th_\Ab$ & $\Ab$ & $\Hopfcc$\\
       $\Th_{\SKB}$ & $\SKB$ & $\HBr$ \\
       $\Th_\DiGrp$ & $\DiGrp$ & $\HDiGrp$\\
       $\Th_\RadRng$ & $\RadRng$ & $\HRadRng$\\
       \hline
   \end{tabular}
   \endgroup
   \end{table}

   Let us notice that models in $\Set$ in the last five rows of the table have in common an underlying group structure. They are, indeed, examples of varieties of $\Omega$-groups. These algebraic varieties were introduced by Higgins in \cite{Higgins} and have been widely investigated, since they provide a unified framework for the study of groups, rings, Lie algebras or Jordan algebras.

   We recall that an $\Omega$-group is a group with a set of operations with positive arity defined on it. If we denote by $u$ the neutral element of the group, every operation $f$ of arity $n$ has to satisfy the following condition:
   $$f(u,\dots,u)=u.$$
   This requirement is equivalent to the fact that the neutral element of the group forms a subalgebra with respect to the operations. 

   Every algebraic variety $\mathcal{V}$ of $\Omega$-groups has a corresponding Lawvere theory, say $\Th$, thus it makes sense to consider its linearized version, namely the category of coalgebraic models of $\Th$. For the sake of brevity, we call throughout the paper ``$\Omega$-Hopf algebra'' a model in $\Coalg$ of a Lawvere theory of an algebraic variety of $\Omega$-groups.

\begin{Def}\label{def omega hopf algebras}
    An \textit{$\Omega$-Hopf algebra} is a product-preserving functor from $\Th$ to $\Coalg$, where $\Th$ is a Lawvere theory corresponding to a variety of $\Omega$-groups. 
    We call \textit{category of $\Omega$-Hopf algebras}, a category of coalgebraic models for a Lawvere theory of $\Omega$-groups.
\end{Def}

All coalgebraic models in the last six rows of Table \ref{Table of examples} are, according to our definition, categories of $\Omega$-Hopf algebras. The purpose of Section \ref{omega groups} of this paper is to start an analysis of this class of categories.
   
\end{subsection}

\end{section}

\begin{section}{From models to coalgebraic models}\label{section classical vs coalgebraic}

Let us fix a Lawvere theory $\Th$ and analyse the relation between the category of models $[\Th,\Set]$ and that of coalgebraic models $[\Th,\Coalg]$.
Given a coalgebra $(C,\Delta,\epsilon)$, we say that $c\in C$ is a \textit{grouplike element} if $\Delta(c)=c\otimes c$ and $\epsilon(c)=1$.
One can define two functors:

\begin{minipage}{0.3\textwidth}
        \begin{align*}
            F\colon \Set& \longrightarrow \Coalg\\
            X& \mapsto \K[X]
        \end{align*}
    \end{minipage}
    \begin{minipage}{0.7\textwidth}
        \begin{align*}
            G\colon \Coalg& \longrightarrow \Set\\
            C& \mapsto G(C)\coloneqq\{\text{grouplike elements of } C\}
        \end{align*}
    \end{minipage}

\vspace{0.3cm}
\noindent where $\K[X]$ is the linear space with basis $X$ and with comultiplication and counit the linear maps extending the assignments $\Delta(x)=x\otimes x$ and $\epsilon(x)=1$ for any $x\in X$, respectively. 
It is easy to show that $F$ is left adjoint to $G$: indeed, there is a natural bijection:
$$\Hom_{\Coalg}(\K[X],C)\cong\Hom_{\Set}(X,G(C)),$$
since a map of coalgebras from $\K[X]$ to $C$ is uniquely determined by the images of elements in $X$ and the image of a grouplike element through a homomorphism of coalgebras is a grouplike element.

This adjunction will now be shown to lift to an adjunction between the categories of functors: 
\begin{figure}[H]
    \centering
    \begin{tikzpicture}
        \node(0) at (-2,0) {$[\Th, \Set]$};
        \node(1) at (2,0) {$[\Th,\Coalg].$};
        \node(2) at (0,0) {$\perp$};
        
        \draw[->] (0) to[out=10,in=170] node[above]{\scriptsize{$F\circ -$}} (1);
        \draw[->] (1) to[out=190,in=350] node[below]{\scriptsize{$G\circ-$}} (0);
    \end{tikzpicture}
\end{figure}
Indeed, the functor $G$ preserves finite products since it is a right adjoint, and the functor $F$ also does so: given two sets $X$ and $Y$, the coalgebra $\K[X]\otimes\K[Y]$ is the vector space with basis $X\times Y$ and comultiplication given by the linear map extending the assignments $\Delta(x\otimes y)=x\otimes y \otimes x\otimes y$, hence it is isomorphic to $\K[X\times Y]$. Furthermore, the functor $F$ preserves the terminal object because it sends the one-element set $\{\star\} $ to the cocommutative coalgebra $\K[\{\star\}]\cong \K$.

\begin{Ex}
If we specialize the above adjunction to the Lawvere theory $\Th_\Grp$ of groups, we reconstruct---as a special case of the previous construction---the well-known adjunction between groups and cocommutative Hopf algebras: 
\begin{figure}[H]
    \centering
    \begin{tikzpicture}
        \node(0) at (-2,0) {$\Grp$};
        \node(1) at (2,0) {$\Hopfc,$};
        \node(2) at (0,0) {$\perp$};
        
        \draw[->] (0) to[out=10,in=170] node[above]{\scriptsize{$\K[-]$}} (1);
        \draw[->] (1) to[out=190,in=350] node[below]{\scriptsize{$\mathcal{G}$}} (0);
    \end{tikzpicture}
\end{figure}
\noindent where the functor $\K[-]$ sends a group $G$ to $\K[G]$ equipped with the Hopf algebra structure with multiplication, unit and antipode induced by the group laws. On the other hand, $\mathcal{G}$ associates to any Hopf algebra the group consisting of its grouplike elements.

\end{Ex}

Let us take a Lawvere theory $\Th$ and the functor $F\circ -$ defined as above.

\begin{Lem}\label{F continuous}
    The functor $F\circ-$ preserves finite limits.
\end{Lem}

\begin{proof}
    We have already seen that the functor $F$ preserves finite products. Thus, $F\circ -$ preserves them too: given two functors $X,Y\in[\Th,\Set]$, we get
    \begin{align*}
        (F\circ X)\times(F\circ Y)(T)&\cong FX(T)\otimes FY(T)\cong \\&\cong\K[X(T)]\otimes\K[Y(T)]\cong\\&\cong \K[X(T)\times Y(T)]\cong\\&\cong F\circ(X\times Y) (T).
    \end{align*}
\noindent Furthermore, the terminal object $S$ in $[\Th,\Set]$---i.e. the functor sending $T$ to the one-element set---is sent to $(F\circ S)(T)\cong F(\{\star\})\cong \K$.
    Let us show that the functor $F$ also preserves equalizers. We take two parallel arrows $f,g\colon I\rightarrow J$ in $\Set$ and their equalizer $E=\{i\in I\mid f(i)=g(i)\}$. By applying the functor $F$ one gets the diagram
    \begin{figure}[H]
        \centering
        \begin{tikzpicture}
            \node(E) at (-2,0) {$\K[E]$};
            \node(I) at (0,0) {$\K[I]$};
            \node(J) at (2,0) {$\K[J].$};

            \draw[->] (0.4,0.1) to node[above]{\scriptsize{$Ff$}} (1.5,0.1);
            \draw[->] (0.4,-0.1) to node[below]{\scriptsize{$Fg$}} (1.5,-0.1);
            \draw[->] (E) to node[]{} (I);
        \end{tikzpicture}
    \end{figure}
    \noindent In $\Coalg$ the equalizer of $F(f)$ and $F(g)$ can be computed as follows (see \cite{AgoreLim}, for instance):
    \begin{align*}
        \mathrm{Eq}(Ff,Fg)=&\{c\in\K[I]\mid c_1\otimes Ff(c_2)=c_1\otimes Fg( c_2)\}\\
        =&\Big\{\sum_{x\in S\subseteq I}\alpha_x x\mid \sum_{x\in S\subseteq I}\alpha_xx\otimes f(x)= \sum_{x\in S\subseteq I}\alpha_xx\otimes g(x)\Big\},
    \end{align*}

    \noindent where $S$ is the finite subset of $I$ of those $x$ that appear with a non trivial coefficient. We want to show that $\K[E]=\mathrm{Eq}(Ff,Fg)$. If we take an element $c=\sum_{x\in S'\subseteq E}\alpha_x x\in \K[E]$, then $f(x)=g(x)$ for every $x\in S'$ (i.e. for any $x$ that appears with a non trivial coefficient in $c$) and so $c\in \mathrm{Eq}(Ff,Fg)$. Conversely, let $c=\sum_{x\in S\subseteq I}\alpha_x x\in \mathrm{Eq}(Ff,Fg)$. Then $\sum_{x\in S\subseteq I}\alpha_xx\otimes f(x)= \sum_{x\in S\subseteq I}\alpha_xx\otimes g(x)$ implies that 
    $$\sum_{x\in S\subseteq I}\alpha_x x\otimes (f(x)-g(x))=0$$
    and so $f(x)=g(x)$ for every $x\in S$. This means that $S\subseteq E$ and $c\in \K[E]$.

    In order to prove that $F\circ-$ preserves equalizers, it is sufficient to check that, given two parallel arrows $\alpha,\beta\colon C\Rightarrow D$ in $[\Th,\Coalg]$, their equalizer is just
    \begin{align*}
        E\colon \Th &\longrightarrow \Coalg\\
        T &\longmapsto \mathrm{Eq}(\alpha_T,\beta_T).
    \end{align*}
    If $\mathrm{Eq}(\alpha_T,\beta_T)$ is closed under every operation $f\colon T^n\rightarrow T$, the subfunctor $E$ of $C$ is well-defined and is the equalizer. Thus, it suffices to show that $Cf(\mathrm{Eq}(\alpha_T,\beta_T)^{\otimes n})\subseteq \mathrm{Eq}(\alpha_T,\beta_T)$. Let us take some elements $c^1,\dots,c^n\in \mathrm{Eq}(\alpha_T,\beta_T)$, namely $c_1^i\otimes \alpha_T(c_2^i)=c_1^i\otimes \beta_T(c_2^i)$ for every $i=1,\dots,n$. The following equalities hold:
    \begin{align*}
        Cf(c&^1\otimes \dots\otimes c^n)_1\otimes \alpha_TCf(c^1\otimes \dots\otimes c^n)_2=\\
        =& Cf(c^1_1\otimes\dots\otimes c^n_1)\otimes\alpha_TCf(c^1_2\otimes \dots\otimes c^n_2)\\
        =& Cf(c^1_1\otimes\dots\otimes c^n_1)\otimes Df((\alpha_T(c^1_2)\otimes \dots\otimes \alpha_T(c^n_2))\\
        =& (Cf\otimes Df)\big((c^1_1\otimes\dots\otimes c^n_1)\otimes (\alpha_T(c^1_2)\otimes \dots\otimes\alpha_T(c^n_2))\big)\\
        =& (Cf\otimes Df)\big((c^1_1\otimes\dots\otimes c^n_1)\otimes (\beta_T(c^1_2)\otimes \dots\otimes\beta_T(c^n_2))\big)\\
        =& Cf(c^1_1\otimes\dots\otimes c^n_1)\otimes Df(\beta_T(c^1_2)\otimes\dots\otimes \beta_T(c^n_2))\\
        =& Cf(c^1_1\otimes\dots\otimes c^n_1)\otimes \beta_T Cf(c^1_2\otimes\dots\otimes c^n_2)\\
        =& Cf(c^1\otimes\dots\otimes c^n)_1\otimes \beta_T Cf(c^1\otimes\dots\otimes c^n)_2.
        \end{align*}
        This argument concludes the proof that $F\circ-$ preserves finite limits.
    \end{proof}

\begin{Cor}\label{F localization}
    The functor $F\circ-$ preserves finite limits and colimits and realizes $[\Th,\Set]$ as a full subcategory of $[\Th,\Coalg]$.
\end{Cor}
\begin{proof}
    The functor preserves finite limits due to Lemma \ref{F continuous} and colimits since it is a left adjoint. Moreover, $F$ is fully faithful: any coalgebra map $f\colon\K[X]\ra\K[Y]$ is uniquely determined by the image of $X$ and restricts to a function $f\colon X\ra Y$, since the image of a grouplike element through a coalgebra homomorphism is a grouplike element. Therefore, $F\circ-$ is fully faithful, because, for any $X,Y\in [\Th,\Set]$, a coalgebra map $f\colon \K[X(T)]\to\K[Y(T)]$ is natural with respect to operations in $\Th$ if and only if its restriction $f\colon X(T)\to Y(T)$ is so.
\end{proof}

\begin{subsection}{Protomodularity}\label{subsection:protomodularity}

Let us assume that the categories $[\Th, \Set]$ and $[\Th,\Coalg]$ are pointed. 

In order to fix the notation, in this subsection we will use bold characters to denote objects and morphisms in the category $[\Th,\Coalg]$ and their unbolded version to denote the underlying coalgebras and coalgebra morphisms. In other words, the coalgebraic model $\mathbf{H}$ in $[\Th,\Coalg]$ sends $T$ to $H$ and, given a natural transformation $\mathbf{f}\colon \mathbf{H}\to \mathbf{H'}$, $f\colon H\to H'$ is its evaluation in $T$. We point out that the terminal object in the category $[\Th,\Coalg]$ is the functor sending $T$ to the field $\K$ and any operation to the identity. If $[\Th,\Coalg]$ is pointed, then there is a unique map from the terminal $\Th$-coalgebra to $\mathbf{H}$: we denote by $u_H\colon \K\to H$ the first component of this natural transformation and by $0_{HH'}\colon H\to H'$ the composition $u_{H'}\circ \epsilon_H$.

The first interesting property that we want to study is the protomodularity of the category $[\Th,\Coalg]$; in our context, it means understanding whether the split short five lemma holds for $\Th$-coalgebras. In the classical case, the protomodularity is a ``Mal'tsev condition'', namely it depends on the existence of some terms satisfying certain equations. In an analogous way, we prove that a category of coalgebraic models of an algebraic theory is protomodular in the sense of Bourn \cite{Bourn1991Normalization} if and only if there are linearized terms satisfying the linearization of the classical equations characterizing protomodular algebraic varieties. This result allows us to deduce that a category of coalgebraic models $[\Th,\Coalg]$ is protomodular if and only if $[\Th,\Set]$ is protomodular as well.

\begin{Prop}\label{Prop protomod}
    If $[\Th,\Coalg]$ is protomodular, then $[\Th,\Set]$ is also protomodular.
\end{Prop}
\begin{proof}
The thesis immediately follows---see e.g. \cite[Section 6, Example 3)]{Bourn1991Normalization}---from the fact that the functor $F\circ -$ preserves limits and reflects isomorphisms by Corollary \ref{F localization}. 
\end{proof}

\begin{Cor}\label{linearized protomod}
    If $[\Th,\Set]$ is protomodular, then there exist a positive integer $n$ and some linearized terms that can be described as follows: for any $\mathbf{H}$ in $[\Th,\Coalg]$, there are morphisms of coalgebras $\alpha_H^i\colon H\otimes H\rightarrow H$ for $i=1,\dots, n$ and $\beta_H\colon H^{\otimes n+1}\rightarrow H$ such that
    \begin{equation}\label{alphai}
        \alpha^i_H\circ \Delta_H=0_{HH} \quad \text{for any }\;\;i=1,\dots,n 
    \end{equation}
    \begin{equation}\label{beta}
        \beta_H\circ(\alpha_H^1\otimes \dots\otimes \alpha^n_H\otimes (\epsilon_H\otimes \mathrm{id_H}))\circ \Delta_{H\otimes H}^{(n)}=\mathrm{id}_H\otimes \epsilon_H.
    \end{equation}
    
\end{Cor}

\begin{proof}
    If $[\Th,\Set]$ is protomodular, then the characterization of protomodular varieties in \cite{BournJanelidze} shows that there exists a positive integer $n$, some binary terms $\alpha^i$ for $i=1,\dots, n$ and an $n+1$-term $\beta$ satisfying the equations:
    \begin{equation*}
        \alpha^i(x,x)=0 \quad \text{for any }\;\;i=1,\dots,n 
    \end{equation*}
    \begin{equation*}
        \beta(\alpha^1(x,y),\dots,\alpha^n(x,y),y)=x.
    \end{equation*}
    
    By linearizing these terms, we get the required linearized terms.
\end{proof}

In particular, from  Proposition \ref{Prop protomod} and Corollary \ref{linearized protomod} one obtains that, if $[\Th,\Coalg]$ is protomodular, then the above-mentioned linearized terms exist. Now we want to show that the converse implication also holds. The proof is similar to the one concerning algebraic varieties (see \cite{BournJanelidze, GranRosicky}).

\begin{Prop}\label{Protomodularity mal'tsev condition}
    Let $[\Th,\Coalg]$ have some binary terms $\alpha^i$ for $i=1,\dots, n$ and an $(n+1)-$ary term $\beta$ that satisfy the linearized equations (\ref{alphai}) and (\ref{beta}) for every $\mathbf{H}$ in $[\Th,\Coalg]$. Then $[\Th, \Coalg]$ is protomodular.
\end{Prop}
\begin{proof}
    Let us consider the following commutative diagram in $[\Th,\Coalg]$:
    \begin{figure}[H]
        \centering
        \begin{tikzpicture}
            \node(A) at (-2,0) {$\mathbf A$};
            \node(B) at (0,0) {$\mathbf B$};
            \node(C) at (2,0) {$\mathbf C$};
            \node(A') at (-2,-2) {$\mathbf{A'}$};
            \node(B') at (0,-2) {$\mathbf{B'}$};
            \node(C') at (2,-2) {$\mathbf{C'}$};

            \draw[->] (A) to node[above]{\scriptsize{$\mathbf{k}$}} (B);
            \draw[->] (B) to node[above]{\scriptsize{$\mathbf p$}} (C);
            \draw[->] (C) to[out=200,in=340] node[below] {\scriptsize{$\mathbf s$}} (B);
            \draw[->] (A') to node[above]{\scriptsize{$\mathbf{k'}$}} (B');
            \draw[->] (B') to node[above]{\scriptsize{$\mathbf{p'}$}} (C');
            \draw[->] (C') to[out=200,in=340] node[below] {\scriptsize{$\mathbf{s'}$}} (B');
            \draw[->] (A) to node[left]{\scriptsize{$\mathbf {f}$}} (A');
            \draw[->] (B) to node[left]{\scriptsize{$\mathbf{g}$}} (B');
            \draw[->] (C) to node[left]{\scriptsize{$\mathbf{h}$}} (C');
            
        \end{tikzpicture}
    \end{figure}
    \noindent where $\mathbf k$ and $\mathbf{k'}$ are the kernels of $\mathbf p$ and $\mathbf{p'}$ respectively, $\mathbf{ps}=\mathbf{id_C}$, $\mathbf{p's'}=\mathbf{id_{C'}}$ and the maps $\mathbf f$ and $\mathbf h$ are isos. We want to prove that $\mathbf g$ is an isomorphism as well. It suffices to prove that $g$ is an iso.
    \vspace{2mm}
    
    \underline{Step 1}: We first show that $g$ has the following property: for any $\mathbf H$ in $[\Th,\Coalg]$ and any pair of coalgebra maps $x,y\colon H\to B$, if there exists a map $\mathbf w\colon \mathbf H\to\mathbf {B'} $ in $[\Th,\Coalg]$ with $w=gx=gy$, then $x=y$. 

    \vspace{2mm}
    \noindent Let $x,y$ and $\mathbf{w}$ be as above. For every map $\mathbf z\colon \mathbf Z\to \mathbf{Z'}$ in $[\Th,\Coalg]$, using the naturality of $\mathbf z$ and the equation \eqref{alphai}, we obtain the following:
    \begin{equation}\label{eq1}
        \alpha_{Z'}^i\circ(z\otimes z)\circ \Delta_Z= z\circ \alpha_Z^i\circ\Delta_Z=0_{ZZ'}\quad \text{for every}\; i=1,\dots,n .
    \end{equation}
    Let us consider the coalgebra maps $\varphi^i\coloneqq \alpha_B^i\circ(x\otimes y)\circ\Delta_H$ where $i=1,\dots,n$. The following equalities hold:

    \noindent\begin{minipage}{0.65\textwidth}
        \begin{align*}
        g\circ \varphi^i= &g\circ \alpha^i_B\circ(x\otimes y)\circ \Delta_H\\
        =& \alpha_{B'}^i\circ (g\otimes g)\circ(x\otimes y)\circ\Delta_H\\
        =&\alpha^i_{B'}\circ(gx\otimes gy)\circ \Delta_H\\
        \overset{gx=gy=w}{=}& \alpha^i_{B'}\circ(w\otimes w)\circ \Delta_H\overset{\ref{eq1}}{=}0_{HB'}.
    \end{align*}
    We get $0_{HC'}=p'\circ g\circ \varphi^i=h\circ p\circ\varphi^i$ which implies $p\circ \varphi^i=0_{HC}$, because $\mathbf h$ is a isomorphism and $\mathbf{h}^{-1}\circ\mathbf{0}_{\mathbf{HC'}}=\mathbf{0}_{\mathbf{HC}}$. Since equalizers in $[\Th,\Coalg]$ are constructed as in $\Coalg$, $k$ is the equalizer of the pair $(p,0_{BC})$ in $\Coalg$. Hence, for every $i=1,\dots,n$, there exists an arrow $\overline{\varphi^i}$ such that $k\circ \overline{\varphi^i}= \varphi^i$.
    \end{minipage}
    \begin{minipage}{0.35\textwidth}
        \begin{figure}[H]
        \centering
        \begin{tikzpicture}
            \node(A) at (-2,0) {$A$};
            \node(B) at (0,0) {$B$};
            \node(C) at (2,0) {$C$};
            \node(A') at (-2,-2) {$A'$};
            \node(B') at (0,-2) {$B'$};
            \node(C') at (2,-2) {$C'$};
            \node(H) at (-1.5,1.5) {$H$};

            \draw[->] (A) to node[above]{\scriptsize{$k$}} (B);
            \draw[->] (B) to node[above]{\scriptsize{$p$}} (C);
            \draw[->] (C) to[out=200,in=340] node[below] {\scriptsize{$s$}} (B);
            \draw[->] (A') to node[above]{\scriptsize{$k'$}} (B');
            \draw[->] (B') to node[above]{\scriptsize{$p'$}} (C');
            \draw[->] (C') to[out=200,in=340] node[below] {\scriptsize{$s'$}} (B');
            \draw[->] (A) to node[left]{\scriptsize{$f$}} (A');
            \draw[->] (B) to node[left]{\scriptsize{$g$}} (B');
            \draw[->] (C) to node[right]{\scriptsize{$h$}} (C');
            \draw[->] (H) to node[above]{\scriptsize{$\varphi^i$}} (B);
            \draw[->,dashed] (H) to node[left]{\scriptsize{$\overline{\varphi^i}$}} (A);
            
        \end{tikzpicture}
    \end{figure}
    \end{minipage}

    \vspace{0mm}
    \noindent Using these maps, we get $$k'\circ f\circ 0_{HA}=0_{HB'}=g\circ \varphi^i=g\circ k\circ\overline{\varphi^i}=k'\circ f\circ\overline{\varphi^i},$$ 
    so $\overline{\varphi^i}=0_{HA}$ because $k'$ is an equalizer---hence a mono---and $f$ is an iso. Thus, 
    \begin{equation}\label{eq2}
        \alpha_B^i\circ(x\otimes y)\circ\Delta_H=\varphi^i=0_{HB}
    \end{equation}
    for every $i=1,\dots,n$.
    
    \noindent Let us observe that $x=(\mathrm{id}_B\otimes \epsilon_B)\circ(x\otimes y)\circ\Delta_H$ since $\mathrm{id}_B\otimes \epsilon_B$ is the first projection from the product $B\otimes B$ and $(x\otimes y)\circ\Delta_H$ is the unique map from $H$ to  $B\otimes B$ induced by $x$, $y$ and the universal property of the product $B \otimes B$.
    By using the linearized equation \ref{beta}, we obtain:
    \begin{align*}
        x=&(\mathrm{id}_B\otimes \epsilon_B)\circ(x\otimes y)\circ\Delta_H\\
        \overset{\ref{beta}}{=}&\beta_B\circ(\alpha_B^1\otimes \dots\otimes \alpha^n_B\otimes (\epsilon_B\otimes \mathrm{id_B}))\circ \Delta_{B\otimes B}^{(n)}\circ(x\otimes y)\circ\Delta_H\\
        =&\beta_B\circ(\alpha_B^1\otimes \dots\otimes \alpha^n_B\otimes (\epsilon_B\otimes \mathrm{id_B}))\circ (x\otimes y)^{\otimes n+1}\circ\Delta_{H\otimes H}^{(n)}\circ\Delta_H\\
        =&\beta_B\circ\Big([\alpha_B^1\circ(x\otimes y)]\otimes \dots\otimes [\alpha^n_B\circ(x\otimes y)]\otimes [(\epsilon_B\otimes \mathrm{id_B})\circ(x\otimes y)]\Big)\circ\Delta_H^{\otimes n+1}\circ\Delta_H^{(n)}\\
        =&\beta_B\circ\Big([\alpha_B^1\circ(x\otimes y)\circ\Delta_H]\otimes \dots\otimes [\alpha^n_B\circ(x\otimes y)\circ\Delta_H]\otimes [(\epsilon_H\otimes y)\circ\Delta_H]\Big)\circ\Delta_H^{(n)}\\
        =&\beta_B\circ\Big(\varphi^1\otimes \dots\otimes \varphi^n\otimes [(\epsilon_H\otimes y)\circ\Delta_H]\Big)\circ\Delta_H^{(n)}\\
        =&\beta_B\circ\Big(0_{HB}\otimes \dots\otimes 0_{HB}\otimes [(\epsilon_H\otimes y)\circ\Delta_H]\Big)\circ\Delta_H^{(n)}\\
        \overset{(!)}{=}&\beta_B\circ\Big([\alpha_B^1\circ(y\otimes y)\circ\Delta_H]\otimes \dots\otimes [\alpha^n_B\circ(y\otimes y)\circ\Delta_H]\otimes [(\epsilon_H\otimes y)\circ\Delta_H]\Big)\circ\Delta_H^{(n)}\\
        =&(\mathrm{id}_B\otimes \epsilon_B)\circ(y\otimes y)\circ\Delta_H=y.
    \end{align*}
    where equality $(!)$ follows from \eqref{eq2} in the special case $x=y$.

    \vspace{2mm}
    \underline{Step 2}: We want to show that $g$ is a split epimorphism. 

    \vspace{2mm}    
    \noindent Let us define $\psi\coloneqq g\circ s\circ h^{-1}\circ p'$ and $\varphi^i\coloneqq \alpha^i_{B'}\circ(\mathrm{id}_{B'}\otimes \psi)\circ\Delta_{B'}$ for every $i=1,\dots,n$.

    \noindent\begin{minipage}{0.55\textwidth}
        We can compute $p'\circ\varphi^i$ for every $i=1,\dots,n$:
        \begin{align*}
            p'\circ\varphi^i=& p'\circ \alpha^i_{B'}\circ(\mathrm{id}_{B'}\otimes \psi)\circ\Delta_{B'}\\           =&\alpha^i_{C'}\circ(p'\otimes p'\psi)\circ\Delta_{B'}\\
            =& \alpha^i_{C'}\circ(p'\otimes p')\circ\Delta_{B'}=0_{B'C'},
        \end{align*}
        where the second last equality holds since $p'\psi=p'gsh^{-1}p'=hpsh^{-1}p'=p'$. Hence, for every $i=1,\dots,n$, there are some maps $\overline{\varphi^i}\colon B'\rightarrow A'$ such that $k'\circ\overline{\varphi^i}=\varphi^i$.
    \end{minipage}
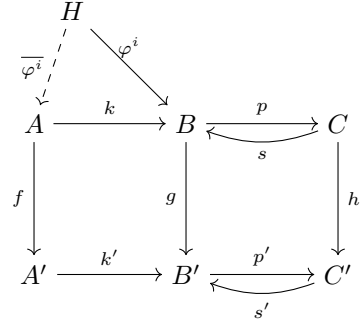
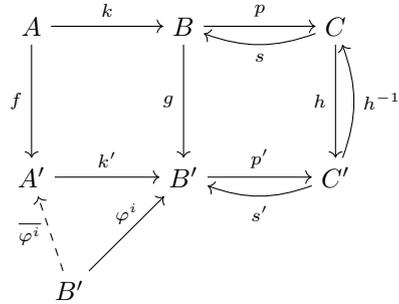
    \begin{minipage}{0.45\textwidth}
        \begin{figure}[H]
            \centering
            \begin{tikzpicture}
            \node(A) at (-2,0) {$A$};
            \node(B) at (0,0) {$B$};
            \node(C) at (2,0) {$C$};
            \node(A') at (-2,-2) {$A'$};
            \node(B') at (0,-2) {$B'$};
            \node(C') at (2,-2) {$C'$};
            \node(H) at (-1.5,-3.5) {$B'$};

            \draw[->] (A) to node[above]{\scriptsize{$k$}} (B);
            \draw[->] (B) to node[above]{\scriptsize{$p$}} (C);
            \draw[->] (C) to[out=200,in=340] node[below] {\scriptsize{$s$}} (B);
            \draw[->] (A') to node[above]{\scriptsize{$k'$}} (B');
            \draw[->] (B') to node[above]{\scriptsize{$p'$}} (C');
            \draw[->] (C') to[out=200,in=340] node[below] {\scriptsize{$s'$}} (B');
            \draw[->] (A) to node[left]{\scriptsize{$f$}} (A');
            \draw[->] (B) to node[left]{\scriptsize{$g$}} (B');
            \draw[->] (C) to node[left]{\scriptsize{$h$}} (C');
            \draw[->] (H) to node[above]{\scriptsize{$\varphi^i$}} (B');
            \draw[->,dashed] (H) to node[left]{\scriptsize{$\overline{\varphi^i}$}} (A');
            \draw[->] (C') to[in=290, out=70] node[right] {\scriptsize{$h^{-1}$}} (C);
            \end{tikzpicture}
        \end{figure}
    \end{minipage}
    
    \vspace{0.2cm}
    \noindent Since $\mathrm{id}_{B'}=(\mathrm{id}_{B'}\otimes\epsilon_{B'})\circ(\mathrm{id}_{B'}\otimes \psi)\circ\Delta_{B'}$, using \ref{beta} we obtain:
    \begin{align*}
        \mathrm{id}_{B'}=& \beta_{B'}\circ(\alpha_{B'}^1\otimes \dots\otimes \alpha^n_{B'}\otimes (\epsilon_{B'}\otimes \mathrm{id}_{B'}))\circ \Delta_{B'\otimes B'}^{(n)}\circ(\mathrm{id}_{B'}\otimes \psi)\circ\Delta_{B'}\\
        =& \beta_{B'}\circ(\alpha_{B'}^1\otimes \dots\otimes \alpha^n_{B'}\otimes (\epsilon_{B'}\otimes \mathrm{id}_{B'}))\circ(\mathrm{id}_{B'}\otimes \psi)^{\otimes n+1}\circ \Delta_{B'\otimes B'}^{(n)}\circ\Delta_{B'}\\
        =& \beta_{B'}\circ\Big([\alpha_{B'}^1\circ(\mathrm{id}_{B'}\otimes \psi)]\otimes \dots\otimes [\alpha^n_{B'}\circ(\mathrm{id}_{B'}\otimes \psi)]\otimes (\epsilon_{B'}\otimes \psi)\Big)\circ\Delta_{B'}^{\otimes n+1}\circ \Delta_{B'}^{(n)}\\
        =& \beta_{B'}\circ(\varphi^1\otimes\dots\otimes\varphi^n\otimes\psi)\circ\Delta_{B'}^{(n)}\\        =&\beta_{B'}\circ(k'\overline{\varphi^1}\otimes\dots\otimes k'\overline{\varphi^n}\otimes\psi)\circ\Delta_{B'}^{(n)}\\        =&\beta_{B'}\circ(gkf^{-1}\overline{\varphi^1}\otimes\dots\otimes gkf^{-1}\overline{\varphi^n}\otimes gsh^{-1}p')\circ\Delta_{B'}^{(n)}\\        =&g\circ \underbrace{\beta_{B}\circ(kf^{-1}\overline{\varphi^1}\otimes\dots\otimes kf^{-1}\overline{\varphi^n}\otimes sh^{-1}p')\circ\Delta_{B'}^{(n)}}_{\eqqcolon g'}.
    \end{align*}

    \noindent Thus, $g$ has a section $g'$ in $\Coalg$. 
    \vspace{0.3cm}
    
    By ``Step 2'', we deduce that $g\circ g'=\id_{B'}$, hence $g\circ g'\circ g=g\circ\id_{B}$. Using ``Step 1'', we obtain the identity $g'\circ g=\id_B$. Therefore $g$ is an isomorphism.
    
\end{proof}

\begin{Thm}\label{Protomodularity}
    $[\Th,\Set]$ is protomodular if and only if $[\Th,\Coalg]$ is protomodular.
\end{Thm}
\begin{proof}
    If $[\Th,\Set]$ is protomodular, then, using Corollary \ref{linearized protomod} together with the previous proposition, we get the protomodularity of $[\Th,\Coalg]$. Due to Proposition \ref{Prop protomod}, the converse implication also holds.
\end{proof}

Theorem \ref{Protomodularity}  implies the protomodularity of all categories in the right column of Table \ref{Table of examples}.

\begin{Rmk}
    We have seen that some properties of algebraic varieties also hold if we take models in $\Coalg$ instead of $\Set$. Even if we are going to continue along that line of finding analogies between the two cases, it is important to notice that, in the coalgebraic framework, not everything works in the same way as the classical one. 

    In particular, one of the most important theorem concerning algebraic varieties---the Birkhoff's variety theorem---fails. A counterexample is given by the full subcategory $\mathsf{GrpHopf}$ of $\Hopfc$ whose objects are group Hopf algebras, namely Hopf algebras generated by grouplike elements. $\mathsf{GrpHopf}$ is closed under products, subalgebras and homomorphic images, but it is not axiomatizable by (linearized) equations. Indeed, it is straightforward to see that every equation holding in $\mathsf{GrpHopf}$, also holds in $\Hopfc$. Take two terms $s$ and $t$ such that $s^H=t^H$ for any $H$ in $\mathsf{GrpHopf}$; thus, for any group $G$, the group Hopf algebra $\K[G]$ satisfies $s^{\K[G]}=t^{\K[G]}$. However, since the functor $\K[-]\colon \Grp\to\Hopfc$ is fully faithful, the equation $s=t$ holds in any group, i.e. the two terms are equal in the Lawvere theory $\Th_{\Grp}$. By definition of $\Hopfc\cong [\Th_\Grp,\Coalg]$, the linearization of $s=t$ holds in any cocommutative Hopf algebra. Hence, the two categories $\mathsf{GrpHopf}$ and $\Hopfc$ cannot be distinguished just regarding their linearized identities. The idea behind this fact is that the difference between the two categories involves the description of the comultiplication, but the identity $\Delta(x)=x\otimes x$ is not the linearization of any equation of the Lawvere theory, since the right-hand side is not a linearized term.

\end{Rmk}

\end{subsection}
\end{section}

\begin{section}{Algebraic functors}\label{section free functor}

A key result about algebraic varieties is the existence of free functors, i.e.\ of left adjoints to any forgetful functor $[\Th,\Set]\ra\Set$. For instance, if $\Th$ is the theory of groups, the free functor associates with any set $X$ the free group generated by $X$.

More generally, a functor of Lawvere theories $R\colon \S\ra\Th$ induces a so-called algebraic functor
$[\Th,\Set]\ra[\S,\Set]$---we can interpret it as a functor forgetting some ``structure'', namely some operations or some axioms in $\Th$---that always has a left adjoint (see \cite[Theorem 3.7.7]{Borceux_1994}, for instance). 
Free functors are a special case of left adjoints to algebraic functors since the forgetful functor is the algebraic functor induced by the unique arrow from the initial Lawvere theory to $\Th$. 

\begin{Def}
    In analogy to the classical case of algebraic varieties, we call an \textit{algebraic functor} a functor between categories of coalgebraic models $[\Th,\Coalg]\ra[\S,\Coalg]$ induced by the precomposition with a functor between the corresponding Lawvere theories.
\end{Def}

In literature, there are several examples of constructions of left adjoints to algebraic functors between coalgebraic models. First of all, that of a free functor from the category of cocommutative coalgebras $\Coalg$ to the category of cocommutative $\Hopfc$ was presented in \cite{Porst09102015} as an application of the more general construction of a free functor from the category of coalgebras to the category of Hopf algebras presented in \cite{Takeuchi1971561}. More recently, Agore and Chirvasitu have proved, in \cite{agore2025categoryhopfbraces}, some similar results concerning the category of Hopf braces: they showed the existence of left adjoints to the forgetful functors $\HBr\ra \Coalg$ and $\HBr\ra\Hopfc$. 

In light of these constructions, the question about the existence, in full generality, of left adjoints to algebraic functors between coalgebraic models naturally arises. 

We recall some results that give us a positive answer to this question. 

Since the category $\Coalg$ of cocommutative coalgebras is locally presentable \cite{Porst2006}, any category of coalgebraic models is locally presentable \cite{Adamek_Rosicky_1994}. 

We can also apply Theorem 1 and Theorem 2 of \cite{PorstFreeInternalGroup} in order to get the existence of left adjoint to any algebraic functor between categories of coalgebraic models and the monadicity of any forgetful functor $[\Th,\Coalg]\to\Coalg$. 

If we look at the following commutative diagram:


\begin{figure}[H]
    \centering
    \begin{tikzpicture}
        \node(0) at (-1,0.8) {$[\Th,\Coalg]$};
        \node(1) at (-1,-0.8) {$[\S,\Coalg]$};
        \node(2) at (1.5,0) {$\Coalg.$};

        \draw[->] (0) to node[left]{\scriptsize{$F$}} (1);
        \draw[->] (0) -- (2);
        \draw[->] (1) -- (2);
    \end{tikzpicture}
\end{figure}


\noindent 
then, the monadicity of any algebraic functor $F\colon [\Th,\Coalg]\to[\S,\Coalg]$ follows by a straightforward application of \cite[Theorem 4.4.4]{Borceux_1994}. For details on how this theorem applies to the present setting see, for instance, Proposition 4 and 5 in \cite{Bourn1991Normalization}.



We summarize these results in the following Theorem.

\begin{Thm}
    Let $\Th$ and $\S$ be two algebraic theories and $R\colon \S\to\Th$ a morphism of algebraic theories that induces an algebraic functor $-\circ R\colon[\Th,\Coalg]\to[\S,\Coalg]$.
    Then the functor $-\circ R$ has a left adjoint and is monadic.
\end{Thm}

Although we already know---by well-known results---the existence of the left adjoint of an algebraic functor,  we now explicitly construct this adjoint because it will be useful in the subsequent discussion.

Let $\Th$ and $\S$ be two algebraic theories with objects $\{T^i\}_{i\in\N}$ and $\{S^i\}_{i\in \N}$. Let us consider a morphism of algebraic theories $R\colon \S\ra \Th$---namely a product-preserving functor sending $T^i$ to $S^i$ for any $i$---that induces an algebraic functor
$$-\circ R\colon [\Th,\Coalg]\longrightarrow [\S,\Coalg].$$

\noindent The left adjoint $L$ to $-\circ R$ 

\begin{figure}[H]
    \centering
    \begin{tikzpicture}
        \node(0) at (-2,0) {$[\Th, \Coalg]$};
        \node(1) at (2,0) {$[\S,\Coalg]$};
        \node(2) at (0,0) {\rotatebox{180}{$\perp$}};
        
        \draw[->] (0) to[out=10,in=170] node[above]{\scriptsize{$-\circ R$}} (1);
        \draw[->] (1) to[out=190,in=350] node[below]{\scriptsize{$L$}} (0);
    \end{tikzpicture}
\end{figure}

\noindent can be described in the following way.

Let us start by defining, for every positive integer $n$, the category $\T{n}$ whose objects are those of the slice category $\Th_{\downarrow T^n}$ and a morphism from an object $(T^m, t)$ to $(T^l,r)$ is a map $f\colon S^m\to S^l$ in $\S$ such that the following diagram commutes: 
\begin{figure}[H]
    \centering
    \begin{tikzpicture}
        \node(1) at (0.2,0) {$T^m$};
        \node(2) at (2.8,0) {$T^l$};
        \node(3) at (1.5,-1.5) {$T^n$};

        \draw[->] (1) to node[above]{\scriptsize{$Rf$}} (2);
        \draw[->] (1) to node[left] {\scriptsize{$t$}} (3);
        \draw[->] (2) to node[right] {\scriptsize{$r$}} (3);
    \end{tikzpicture}
\end{figure}

We fix a model in $[\S,\Coalg]$

\begin{align*}
    \C \colon \;\S\;&\xrightarrow{\hspace*{1.5cm}}  \Coalg\\
     S^n &\xmapsto{\hspace*{1.5cm}} C^{\otimes n}\\
     (f\colon S^n\mapsto S^m) &\xmapsto{\hspace*{1.5cm}} (f^C\colon C^{\otimes n}\mapsto C^{\otimes m})\\
\end{align*}


\noindent and we take, for every positive integer $n$, the diagram
\begin{align*}
    D(n) \colon \;\T{n}\;&\xrightarrow{\hspace*{1.5cm}}  \Coalg\\
     (T^m,t) &\xmapsto{\hspace*{1.5cm}} C(n)_t\coloneqq C^{\otimes m}\\
     (f\colon (T^m,t)\mapsto (T^l,r)) &\xmapsto{\hspace*{1.5cm}} (f^C\colon C^{\otimes m}\mapsto C^{\otimes l}).\\
\end{align*}

From now on, we write for short $D$ and $C_t$ in place of $D(1)$ and $C(1)_t$.

\noindent Notice that, since the category $\Coalg$ has small colimits \cite{Porst2006}, the colimit $\colim_{D(n)} C(n)_t$ does exist in $\Coalg$.

We prove the following lemma, which simplifies the description of $L$.

\begin{Lem}\label{lemma:product}
    With the same notation as above, the following isomorphisms hold 
    $$\big(\colim_D C_t\big)^{\otimes n}\cong \colim_{D\times \cdots\times D} C_{t_1}\otimes \cdots \otimes C_{t_n}\cong \colim_{D(n)} C(n)_t.$$
\end{Lem}

\noindent We can now define a functor $L\C$:
\begin{align*}
    L\C \colon \;\Th\;&\xrightarrow{\hspace*{1.5cm}}  \Coalg\\
     T^n &\xmapsto{\hspace*{1.5cm}} \colim_{D(n)}C(n)_t\eqqcolon L\C(T^n)\\
     (g\colon T^n\mapsto T^k) &\xmapsto{\hspace*{1.5cm}} \big(g^{L\C}\colon L\C(T^n)\mapsto L\C(T^k)\big),\\
\end{align*}
where $g^{L\C}$ is the unique map that makes the following diagram commute for every $t\colon T^m\ra T^n$
\begin{figure}[H]
        \centering
        \begin{tikzpicture}
            \node(1) at (-2,2) {$C(n)_s$};
            \node(2) at (2,2) {$\colim_{D(n)} C(n)_t$};
            \node(3) at (-2,1.5) {\rotatebox{90}{$=$}};
            \node(4) at (-2,1) {$C^{\otimes m}$};
            \node(5) at (-2,0.5) {\rotatebox{90}{$=$}};
            \node(6) at (-2,0) {$C(k)_{gs}$};
            \node(7) at (2,0) {$\colim_{D(k)}C(k)_u.$};

            \draw[->,dashed] (2) to node[right]{\scriptsize{$g^{L\C}$}}(7);
            \draw[->] (1)--(2);
            \draw[->] (6)--(7);            
        \end{tikzpicture}
    \end{figure}

    \noindent In order to get the idea of the construction of the functor $L\mathbf C$, we can look at its underlying coalgebra $L\C(T)$. It is obtained by adding a copy of $C^{\otimes n}$ for any new $n$-ary term in $\Th$. Since we also need to add these copies in a way that is compatible with axioms in $\S$, it is not enough to take the coproduct of all these copies but we have to consider the colimit of the diagram $D(1)$. We refer to Section \ref{section full functors} for explicit examples of this construction.

\noindent Using Lemma \ref{lemma:product} it is easy to see that the functor $L\C$ preserves finite products, so it lies in the category $[\Th,\Coalg]$. Hence, we can define:
\begin{align*}
    L \colon \;[\S,\Coalg]\;&\xrightarrow{\hspace*{1.5cm}}  [\Th,\Coalg]\\
     \C\;&\xmapsto{\hspace*{1.5cm}} \; L\C\\
     (\alpha\colon \C\Rightarrow \C') &\xmapsto{\hspace*{1.5cm}} (L\alpha\colon L\C\Rightarrow L\C'),\\
\end{align*}
where $L\alpha_T\colon \colim_D C_t\ra \colim_D C'_t$ is the unique map such that, for every $s\colon T^m\ra T^n$, the following diagram commutes.

\begin{figure}[H]
        \centering
        \begin{tikzpicture}
            \node(0) at (-2,2) {$C_s$};
            \node(1) at (2,2) {$\colim_D C_t$};
            \node(2) at (-2,1.5) {\rotatebox{90}{$=$}};
            \node(3) at (-2,1) {$C^{\otimes m}$};
            \node(4) at (-2,-0.5) {$C'^{\otimes m}$};
            \node(5) at (-2,-1) {\rotatebox{90}{$=$}};
            \node(6) at (-2,-1.5) {$C'_s$};
            \node(7) at (2,-1.5) {$\colim_D C'_t$};

            \draw[->] (0)--(1);
            \draw[->] (3) to node[left]{\scriptsize{$\alpha_S^{\otimes m}=\alpha_{S^m}$}} (4);
            \draw[->] (6)--(7);
            \draw[->,dashed] (1) to node[right]{\scriptsize{$L\alpha_T$}}(7); 
        \end{tikzpicture}
    \end{figure}

    We then have the following Lemma.

    \begin{Lem}
        The functor $L$ described above is well-defined and it is left adjoint to $-\circ R$.
    \end{Lem}

    \begin{proof}
        We have to show that, given $\C\colon \S\ra \Coalg$ and $\mathbf H\colon \Th \ra\Coalg$, there is a natural isomorphism
        \begin{equation*}
            \hom_{[\Th,\Coalg]}(L\C,\mathbf H)\cong \hom_{[\S,\Coalg]}(\C,\mathbf H\circ R).
        \end{equation*}
        
        If $\alpha\colon \C\Rightarrow \mathbf H\circ R$ is a map in $[\S,\Coalg]$, we can define a map $\overline{\alpha}\colon L\C\Rightarrow\mathbf H$ in $[\Th,\Coalg]$ by taking $\overline{\alpha}_T\colon\colim_D C_t\ra H$ as the unique map that makes, for every $(T^n,t)$ in $\T1$, the following diagram commutative.
        \begin{figure}[H]
        
            \centering
            \begin{tikzpicture}
                \node(0) at (-2,0) {$C_t$};
                \node(1) at (-2,-0.4) {\rotatebox{90}{$=$}};
                \node(2) at (-2,-0.8) {$C^{\otimes n}$};
                \node(3) at (-2,-2) {$H^{\otimes n}$};
                \node(4) at (2,0) {$\colim_D C_t$};
                \node(5) at (2,-2) {$H$};

                \draw[->] (0)--(4);
                \draw[->] (2) to node[left]{\scriptsize{$\alpha_S^{\otimes n}$}} (3);
                \draw[->] (3) to node[below]{\scriptsize{$t^H$}} (5);
                \draw[->,dashed] (4) to node[right]{\scriptsize{$\overline{\alpha}_T$}} (5);
            \end{tikzpicture}
        \end{figure}
        One can easily check the naturality of $\alpha$.

        This association defines a map
        $\Psi\colon\hom_{[\S,\Coalg]}(\C,\mathbf H\circ R) \to \hom_{[\Th,\Coalg]}(L\C,\mathbf H) $ sending $\alpha$ to $\overline{\alpha}$.
         We want to show that $\Psi$ is a bijection. It is clearly injective since, if $\overline{\alpha}=\overline{\beta}$, then $\alpha=\iota\circ\overline{\alpha}_T=\iota\circ\overline{\beta}_T=\beta$,  where $\iota$ is the inclusion $\iota\colon C_{\id}\ra\colim_D C_t$. In order to prove surjectivity, we take $\theta\colon L\C\Rightarrow \mathbf H$ and we define $\alpha\colon\C\Rightarrow\mathbf H\circ R$ by taking $\alpha_S=\theta_T\circ\iota$. $\alpha$ is a natural transformation in $[\S,\Coalg]$, indeed for every $f\colon S^n\ra S$ we have the diagram
        \begin{figure}[H]
            \centering
            \begin{tikzpicture}
                \node(0) at (-4,2) {$C^{\otimes n}$};
                \node(1) at (-4,0) {$C$};
                \node(2) at (0,2) {$(\colim_D C_t)^{\otimes n}$};
                \node(3) at (0,0) {$\colim_D C_t$};
                \node(4) at (4,2) {$H^{\otimes n}$};
                \node(5) at (4,0) {$H$,};

                \draw[->] (0) to node[below]{\scriptsize{$\iota^{\otimes n}$}} (2);
                \draw[->] (1) to node[above]{\scriptsize{$\iota$}} (3);
                \draw[->] (0) to node[left]{\scriptsize{$f^C$}} (1);
                \draw[->] (2) to node[left]{\scriptsize{$f^{L\C}$}} (3);
                \draw[->] (2) to node[below]{\scriptsize{$\theta_T^{\otimes n}$}} (4);
                \draw[->] (4) to node[right]{\scriptsize{$Rf^H=f^H$}} (5);
                \draw[->] (3) to node[above]{\scriptsize{$\theta_T$}} (5);
                \draw[->] (0) to [out=15,in=165] node[above]{\scriptsize{$\alpha_S^{\otimes n}$}} (4);
                \draw[->] (1) to [out=-15,in=-165] node[below]{\scriptsize{$\alpha_S$}} (5);
            \end{tikzpicture}
        \end{figure}
        \noindent where the square on the left commutes by definition of the colimit diagram and the right-hand square commutes by the naturality of $\theta$. 

        The last thing to be proven is that $\theta=\overline{\alpha}$, namely that $\theta_T=\overline{\alpha}_T\colon \colim_D C_t\ra H$. This is the same as proving that, for every $t\colon T^n\ra T$, the diagram in the following picture commutes.
        \begin{figure}[H]
            \centering
            \begin{tikzpicture}
                \node(0) at (-3,1) {$C_t$};
                \node(1) at (-3,0.5) {\rotatebox{90}{$=$}};
                \node(2) at (-3,0) {$C^{\otimes n}$};
                \node(3) at (0,1) {$\colim_D C_t$};
                \node(4) at (0,0) {$H^{\otimes n}$};
                \node(5) at (3,0.5) {$H$};

                \draw[->] (0)--(3);
                \draw[->] (2) to node[below]{\scriptsize{$\alpha_S^{\otimes n}$}} (4);
                \draw[->] (3) to node[above]{\scriptsize $\theta_T$} (5);
                \draw[->] (4) to node[below]{\scriptsize{$t^H$}} (5);
            \end{tikzpicture}
        \end{figure}
        \noindent But it follows from the fact that in the figure 
        \begin{figure}[H]
            \centering
            \begin{tikzpicture}
                \node(0) at (-4,2) {$C^{\otimes n}$};
                \node(u) at (-4,1.5) {\rotatebox{90}{$=$}};
                \node(1) at (-4,1) {$C_t$};
                \node(2) at (0,2) {$(\colim_D C_t)^{\otimes n}$};
                \node(3) at (0,0) {$\colim_D C_t$};
                \node(4) at (4,2) {$H^{\otimes n}$};
                \node(5) at (4,0) {$H$,};

                \draw[->] (0) to node[below]{\scriptsize{$\iota^{\otimes n}$}} (2);
                \draw[->] (1) to node[above]{\scriptsize{$\iota$}} (3);
                \draw[->] (2) to node[left]{\scriptsize{$t^{L\C}$}} (3);
                \draw[->] (2) to node[below]{\scriptsize{$\theta_T^{\otimes n}$}} (4);
                \draw[->] (4) to node[right]{\scriptsize{$t^H$}} (5);
                \draw[->] (3) to node[above]{\scriptsize{$\theta_T$}} (5);
                \draw[->] (0) to [out=15,in=165] node[above]{\scriptsize{$\alpha_S^{\otimes n}$}} (4);
            \end{tikzpicture}
        \end{figure}

        \noindent the left-hand triangle commutes by definition of $t^{L\C}$ and the right-hand one by naturality of $\theta$.

        This ends the proof that $\Psi$ is a bijection, leading to the required adjunction.
    \end{proof}

    The existence of left adjoint to algebraic functors for categories of coalgebraic models specializes in Takeuchi's construction \cite{Takeuchi1971561} in the cocommutative case. For this, take $\S$ to be the initial theory, namely the trivial Lawvere theory of finite non-empty sets and $\Th$ to be the  theory of groups $\Th_\Grp$. Under these assumptions, the algebraic functor becomes the forgetful functor $\Hopfc\to\Coalg$, whose left adjoint is, up to isomorphism, the free functor described by Takeuchi in \cite{Takeuchi1971561}.

    Similarly, we recover the result about the existence of free functors $\Coalg\to\HBr$ and $\Hopfc\to\HBr$ proved in \cite{agore2025categoryhopfbraces} by taking the theory of skew braces $\Th_\SKB$ and the algebraic functors obtained by precomposition of the initial functor with  the functor $\Th_\SKB\to\Th_\Grp$. The left adjoints to the so obtained algebraic functors are, up to iso, the required free functors. Notice that we can choose two different functors $\Th_\SKB\to\Th_\Grp$, because in the theory of skew braces there are two different group structures and we can decide which one we want to forget.

    
\end{section}    

\begin{section}{Colimits}\label{Section Colimts}

We have just seen that every category of coalgebraic models of a Lawvere theory has a forgetful functor to $\Coalg$ that is a right adjoint; it follows that limits in a category $[\Th,\Coalg]$ exist and are computed as in $\Coalg$. 

What about colimits?
It is well known that the category of cocommutative Hopf algebras is complete and cocomplete  \cite{Sweedler,Porst2011,AgoreLim} and the same holds for the category of Hopf braces \cite{AgoreConstructingHB,agore2025categoryhopfbraces}.
Hence, it is natural to ask ourselves if colimits can be constructed in any category of coalgebraic models. In fact, we now provide a description for colimits in any such category.

\begin{subsection}{Coequalizers}\label{coequalizer construction}

Let us fix a Lawvere theory $\Th$. 

\begin{Def}
    Given $\mathbf{A}\in[\Th,\Coalg]$, a \textit{coideal} of $A$ is a linear subspace $I\subset A$ such that
    \begin{enumerate}
        \item $\Delta(I)\subset I \otimes A+A\otimes I$,
        \item $\epsilon(I)=0$.
    \end{enumerate}
    We call a $\Th$\textit{-ideal} a subset $I\subset A$ that is an ideal with respect to the operations in $\Th$, namely for every operation $t\colon T^n\ra T$ it holds that $t^A(a_1\otimes\cdots\otimes a_n)\in I$ whenever there exists an index $i$ such that $a_i\in I$.   
    
\end{Def}

It is straightforward to check the following remarks.

\begin{Rmk}
    If $I$ is a coideal and a $\Th$-ideal of $A$, the linear quotient $A/I$ is a $\Th$-coalgebra and the quotient map $A\ra A/I$ is a map of $\Th$-coalgebras.
\end{Rmk}

\begin{Rmk}\label{ker ideal}
    Let us assume $[\Th,\Coalg]$ to be pointed. 
    Given a map of coalgebraic models of $\Th$, say $f\colon A\ra B$, the linear subspace $\ker f$ given by the vector space kernel of $f$ is a $\Th$-ideal (but not a coalgebra in general) of $A$. Indeed, for any term $t\colon T^n\ra T$ and every elements $a\in \ker f$, $a_1,\dots,a_{n-1}\in A$, we have $$f(t^A(a\otimes a_1\otimes\cdots\otimes a_{n-1}))=t^B(0^B\otimes f(a_1)\otimes\cdots \otimes f(a_{n-1}))=t(0^{B^{\otimes n}})=0^B.$$
\end{Rmk}

\begin{Prop}\label{T ideal generated by a coideal}
    Let $A\in[\Th,\Coalg]$ and let $I$ be a coideal of $A$. Then, the $\Th$-ideal generated by $I$---namely the smallest $\Th$-ideal containing $I$---exists and it is also a coideal. We denote it by $\langle I\rangle_\Th$.
\end{Prop}

\begin{proof}
    Let us define the following subsets of $A$ by recursion:
    \begin{itemize}
        \item $I_0\coloneqq I$;
        \item {$I_j$ is the linear subspace generated by $$I_{j-1}\cup \{t^A(x_1\otimes \cdots\otimes x_n)\mid n\in\N,\; t\colon T^n\ra T,\; x_1,\dots,x_n\in A,\; \exists i \text{ s.t. }x_i\in I_{j-1}\}. $$}
    \end{itemize}
    We take $\overline{I}\coloneqq\bigcup_{j\in\N}I_j$. It is clearly the smallest ideal containing $I$. Let us check by induction that every $I_j$ is a coideal. $I=I_0$ is a coideal by assumption. We assume that $I_j$ is a coideal and show it for $I_{j+1}$. It is enough to check that for any generator $x$ of $I_{j+1}$ as linear subspace, $\Delta(x)\in I_{j+1}\otimes A+A\otimes I_{j+1}$. If $x\in I_j$, $\Delta(x)\in I_j\otimes A+ A\otimes I_j\subseteq I_{j+1}\otimes A+A\otimes I_{j+1}$, otherwise we can assume, without loss of generality, that $x$ is of the form $x=t^A(y\otimes x_1\otimes \cdots\otimes x_n)$ for $t\colon T^{n+1}\ra T$, $y\in I_j$ and $x_1,\dots,x_n\in A$. Since $\Delta(y)\in I_j\otimes A+A\otimes I_j$, we get 
    \begin{align*}
        \Delta(x)&=\Delta(t^A(y\otimes x_1\otimes \cdots\otimes x_n))=(t^A\otimes t^A)(\Delta_{A^{\otimes(n+1)}}(y\otimes\cdots\otimes x_n))\\
        &\in t^A(I_j\otimes A\otimes\cdots\otimes A)\otimes t^A(A\otimes\cdots \otimes A)+t^A(A\otimes\cdots\otimes A)\otimes t^A(I_j\otimes A\otimes\cdots\otimes A)\\
        &\subseteq I_{j+1}\otimes A+ A\otimes I_{j+1},
    \end{align*}
    i.e. $I_{j+1}$ is a coideal.
    Hence $\overline{I}$ is the $\Th$-ideal generated by $I$.
\end{proof}

\begin{Prop}
    Let $A\xbigtoto[f]{g} B$ be two maps in $[\Th,\Coalg]$. The coequalizer of $f$ and $g$ is given by the linear quotient $q\colon B\ra \faktor{B}{\langle I\rangle_{\Th}}$ where 
    $I\coloneqq \{f(a)-g(a)\mid a\in A\}$.
\end{Prop}
\begin{proof}
    The subspace $I$ is clearly a coideal, thus, due to the previous proposition, $\faktor{B}{\langle I\rangle_{\Th}}$ is a $\Th$-coalgebra and so it corresponds to a functor in $[\Th,\Coalg]$. The map $q$ coequalizes $f$ and $g$; furthermore, given another map $h\colon B\ra C$ such that $hf=hg$, it holds that $I\subseteq \ker h$. Since $\ker h$ is a $\Th$-ideal, we moreover have that $\langle I\rangle_{\Th}\subseteq \ker h$ and, thus, the map $h$ factors through to a unique map $\overline{h}\colon \faktor{B}{\langle I\rangle_{\Th}}\ra C$.   
\end{proof}

\begin{Rmk}\label{Cokernel}
    Let us assume $[\Th,\Coalg]$ to be pointed, and denote by $u_A\colon \K\to A$ the unique map from the zero object $\K$ to a $\Th$-coalgebra $A$. In order to compute the cokernel of a map $f\colon A\to B$ in $[\Th,\Coalg]$, we should quotient by the $\Th$-ideal generated by $I\coloneqq \{ f(x)-u_B\epsilon_A(x)\mid x\in A\}$, where $\epsilon_A$ is the counit of $A$. A more explicit description of $I$ is the following:
    $$I=\{y\in f[A]\mid \epsilon_B(y)=0\}\eqqcolon f[A]^+.$$
    Indeed, take an element of the form $f(x)-u_B\epsilon_A(x)\in I$. Then $$f(x)-u_B\epsilon_A(x)=f(x)-f(u_A\epsilon_A(x))=f(x-u_A\epsilon_A(x))\in f[A]$$ and $$\epsilon_B(f(x)-u_B\epsilon_A(x))=\epsilon_Bf(x)-\epsilon_Bu_B\epsilon_A(x)=\epsilon_A(x)-\epsilon_A(x)=0,$$ namely the element lies in $f[A]^+$. Conversely, if $f(x)=y\in f[A]^+$, from $\epsilon_B(y)=0$ it follows that $\epsilon_Bf(x)=\epsilon_A(x)=0$ and $u_B\epsilon_A(x)=0$, thus $y=f(x)-u_B\epsilon_A(x)\in I$.
    As a consequence, the cokernel of $f$ is the quotient $\faktor{B}{\langle f[A]^+\rangle}_\Th$.
\end{Rmk}

\begin{Rmk}
    Let us spell out how the construction of coequalizers in $\Hopfc$ (for their description see \cite{Andruskiewitsch}) can be seen as a special case of this construction. Let us take $f,g\colon A\to B$ maps in $\Hopfc$ and $I\colon \{f(a)-g(a)\mid a\in A\}$. In order to understand the coequalizer of $f$ and $g$, we should study the ideal generated by the coideal $I$. Denoted by $\cdot$ and $S$ the multiplication and the antipode in $B$ respectively, we consider the linear subspace $BIB$ generated by the set $\{b\cdot i\cdot b'\mid b,b'\in B,\;i\in I\}$; it is trivially closed under the multiplication by elements in $B$ and is also closed under the antipode, since 
    $$S(b\cdot (f(a)-g(a))\cdot b')=S(b')\cdot (f(S(a))-g((S(a)))\cdot S(b).$$
    Hence, $BIB=\langle I\rangle_{\Th_{\Grp}}$ is the $\Th_\Grp$-ideal generated by $I$ and the coequalizer is given by $\faktor{B}{BIB}$.
\end{Rmk}
   
\end{subsection}

\begin{subsection}{Coproducts}\label{coproducts construction}

Let us adopt the same notation of Subsection  \ref{subsection:protomodularity}.
The construction of coproducts presented in this subsection is inspired by that of coproducts in the category $\HBr$ of cocommutative Hopf braces described in \cite{agore2025categoryhopfbraces}.

Let $\{\mathbf A_i\}_{i\in I}$ be a family of coalgebraic models for $\Th$ and let us consider the forgetful functor - free functor adjunction
\begin{figure}[H]
    \centering
    \begin{tikzpicture}
        \node(0) at (-2,0) {$[\Th, \Coalg]$};
        \node(1) at (1.7,0) {$\Coalg.$};
        \node(2) at (0,0) {\rotatebox{180}{$\perp$}};
        
        \draw[->] (0) to[out=10,in=170] node[above]{\scriptsize{$U$}}  (1);
        \draw[->] (1) to[out=190,in=350] node[below]{\scriptsize{$F$}} (0);
    \end{tikzpicture}
\end{figure}
Our aim is to construct coproducts in $[\Th,\Coalg]$.

Let $(C, \iota_i\colon A_i\ra C)$ be the coproduct of $\{A_i\}_{i\in I}$ in $\Coalg$, i.e. $C=\bigsqcup_{i\in I}A_i$, and take its image $FC$ in $[\Th,\Coalg]$ under the free functor $F$. There are coalgebra maps $\alpha_i\coloneqq\eta_C\circ \iota_i\colon A_i\ra UFC$, where $\eta$ is the unit of the adjunction. Since the maps $\alpha_i$ do not lie in $[\Th,\Coalg]$ in general, we want to take a suitable quotient of $UFC$ that makes the $\alpha_i$ preserve the operations in $\Th$.
Denote by $\{f_j\colon T^{n_j}\ra T\}$ the operations in $\Th$. There are coideals of $UFC$
$$I_j\coloneqq\{\alpha_i(f^
{A_i}_j(x_1\otimes\cdots\otimes x_
{n_j}))-f_j^{FC}(\alpha_i(x_1)\otimes\cdots\otimes\alpha_i(x_{n_j}))\mid x_1,\dots,x_{n_j}\in A_i,i\in I\}.$$
Let us consider the linear subspace $I$ generated by the union $\bigcup_{j\in J}I_j$. $I$ is a coideal, so, using Proposition \ref{T ideal generated by a coideal}, we can take the $\Th$-ideal and coideal $\langle I\rangle_\Th$ generated by $I$ and consider the coalgebra $A\coloneqq UFC/\langle I\rangle_\Th$ with quotient map $q\colon UFC\ra UFC/\langle I\rangle_\Th$. Since $UFC$ has the structure of $\Th$-coalgebra coming from $FC$ and $\langle I\rangle_\Th$ is a coideal and $\Th$-ideal, $A$ is a $\Th$-coalgebra, i.e. it induces a functor $\mathbf A$ in $[\Th,\Coalg]$ and $q$ is a map of $\Th$-coalgebras, namely a natural transformation in $[\Th,\Coalg]$.
We obtain maps of coalgebras $u_i\coloneqq q\circ\alpha_i\colon A_i\ra A$ that lift to maps of $\Th$-coalgebras, by definition of $A$.
Let us show that $(\mathbf A,u_i\colon \mathbf A_i\ra \mathbf A)$ is the coproduct of the family $
\{\mathbf A_i\}_{i\in I}$ in the category $[\Th,\Coalg]$.

Take $\mathbf B$ a coalgebraic model together with maps $b_i\colon \mathbf A_i\ra \mathbf B$ in $[\Th,\Coalg]$. Since $C$ is the coproduct in $\Coalg$, there is a map of coalgebras $\varphi\colon C\ra B$ such that $\varphi\circ \iota_i=b_i$. By adjunction, we get a map $\overline{\varphi}\colon FC\ra \mathbf B$ in $[\Th,\Coalg]$  satisfying $U\overline{\varphi}\circ \alpha_i=U\overline{\varphi}\circ\eta_C\circ \iota_i=\varphi\circ\iota_i=Ub_i$. Hence, using the universal property of the quotient $\mathbf A$, we get a map $\mathbf A\ra \mathbf B$ because $\overline \varphi$ sends to zero every element in $I$, since $b_i$ is a map of $T$-coalgebras. This means that $\mathbf A$ has the universal property of the coproduct in $[\Th,\Coalg]$. 
    
\end{subsection}

\end{section}

\begin{section}{Surjective maps of theories}\label{section full functors}

Let us consider two Lawvere theories $\S$ and $\Th$ with objects $\{S_i\}_{i\in\N}$ and $\{T_i\}_{i\in\N}$ and a morphism of theories between them, say $R\colon \S\ra \Th$. In this section, we want to study the case where the functor $R$ is a ``surjective map of theories'', namely when it is surjective on operations, i.e. it is a full functor. 
Informally, this means that $\Th$ has no more operations than $\S$, even if it could potentially have more commutative diagrams, i.e. more axioms. In the classical case, having a surjective map of theories means that the corresponding algebraic functor determines a Birkhoff subcategory \cite{Lawvere2008}. In this section, we want to analyze what happens for coalgebraic models.

First of all, let us notice that the fullness of $R$ implies that the induced algebraic functor between the categories of coalgebraic models
$$-\circ R\colon [\Th,\Coalg]\longrightarrow[\S,\Coalg]$$
is fully faithful. Indeed, for any natural transformation $\alpha\in\hom_{
[\S,\Coalg]
}(\mathbf C\circ R,\mathbf D\circ R)$, there is a natural transformation $\beta$ in $\hom_{
[\Th,\Coalg]
}(\mathbf C,\mathbf D)$ such that $\alpha=\beta \circ R$. Indeed $\beta$ can be defined as $\beta_{T^n}\coloneqq \alpha_S^n\colon \mathbf C(T^n)\to \mathbf D(T^n)$ and its naturality follows from the naturality of $\alpha$ and the fact that $R$ is full, i.e. any operation in $\Th$ is the image through $R$ of an operation in $\S$.

It is the case, for instance, of the forgetful functors $\Hopfcc\ra\Hopfc$ and $\HRadRng\ra\HBr\ra \HDiGrp$.

In this section, we show that the fullness of $R$ guarantees an easier description of the free functor (i.e the left adjoint to $-\circ R$), which we provide below.

\begin{Prop}\label{free construction for full functors}
    Let $R\colon\S\ra\Th$ be a surjective map of theories. Then, the free functor $L\colon [\S,\Coalg]\ra[\Th,\Coalg]$ can be described as a coequalizer in the category $[\S,\Coalg]$.
\end{Prop}
\begin{proof}
    We want to specialize the construction of Section \ref{section free functor}. 
    Let $\C\colon \S\ra\Coalg$ be a coalgebraic model of $\S$ with $C\coloneqq\C(T)$, we recall that $L\C$ is given by $L\C(T)\coloneqq \colim_{D} C_t$. We can take the set $\{Rt_i=Rs_i\}_{i\in I}$ of all axioms holding in $\Th$, where $t_i,s_i\colon S^{n_i}\ra S$. Since $R$ is full, from now on, we denote with the same symbol both an operation in $\S$ and its image under $R$. Let us consider the coproduct in $[\S,\Coalg]$ of all $C_{t_i}\cong C^{\otimes n_i}\cong C_{s_i}$ (it exists due to the construction provided in \ref{coproducts construction}) and the diagram
    \begin{figure}[H]
        \centering
        \begin{tikzpicture}
            \node(0) at (-2,0.5) {$\bigsqcup_{i\in I}C_{t_i}$};
            \node(1) at (-2,-0.5) {$\bigsqcup_{i\in I}C_{s_i}$};
            \node(2) at (1,0) {$C$};
            \node(3) at (-2,0) {\rotatebox{90}{$\cong$}};

            \draw[->] (0) to[out=0, in=160] node[above]{\scriptsize{$t$}} (2);
            \draw[->] (1) to[out=0, in=200] node[below]{\scriptsize{$s$}} (2);
            
        \end{tikzpicture}
    \end{figure}
    \noindent where $t$ and $s$ are the map induced respectively by $\{t_i^C\}_{i\in I}$ and $\{s_i^C\}_{i\in I}$, using the universal property of the coproduct. We denote by $(\C',q\colon \C\ra \C')$ the coequalizer in $[\S,\Coalg]$ of $(t,s)$ which defines a functor that lies, in fact, in $[\Th,\Coalg]$. Indeed, the following equivalences hold
    \begin{align*}
        qt=qs\iff q t_i^C=q s_i^C\iff t_i^{C'}q^{\otimes n_i}=s_i^{C'}q^{\otimes n_i}\iff t_i^{C'}=s_i^{C'},
    \end{align*}
    where the second ``if and only if'' follows from the naturality of $q$ and the third one holds since $q$ is surjective and so $q^{\otimes n_i}$ is an epimorphism.

    Our claim is that $\C'$ is isomorphic to $L\C$. 
    
    Denote by $\eta_r$ the composition $C_r\cong C^{\otimes n}\xlongrightarrow{r^C} C\xlongrightarrow{q} C' $. The family $\{\eta_r\mid r\colon S^n\to S\}$ defines a cocone: if $r=l\circ Rf$ for some $f\colon S^n\ra S^m$ and $l\colon S^m\to S$, by definition of $\C'$, $q$ coequalizes the terms $l^Cf^C$ and $r^C$, i.e. $ql^Cf^C=qr^C$. Hence, there is a unique map $L\C\ra \C'$.

    Conversely, since $L\C$ is the colimit over $D$, for any $i\in I$, we have
    $$\iota_{\id} s_i^C=\iota_{s_i}=\iota_{t_i}=\iota_{\id}t_i^C$$
    and so $\iota_\id t=\iota_\id s$. By the universal property of $\C'$ we get a unique map $\C'\ra L\C$ that allows us to conclude the proof.    
\end{proof}

The just obtained description of the free functor agrees with the intuitive idea that, in order to ensure the validity of certain axioms in a universal way, we need to consider a suitable quotient of the coalgebraic model, in analogy with the case of classical algebraic varieties.
In other words, the previous proposition tells us that, when we impose further linearized axioms to a certain category of coalgebraic models (without adding new operations), the free functor from the old category to the new one is obtained by quotienting out the ideal generated by all elements of the form $t(a)-s(a)$ for every new axiom of the form $t=s$.

This procedure applies, for instance, to the algebraic functor $G\colon \HRadRng\to\HBr$. Indeed the functor $\Th_{\HBr}\to\Th_{\HRadRng}$ is full, since every operation in $\HRadRng$ is the image of an operation in $\HBr$. Let us explicitly describe its left adjoint $F\colon \HBr\to\HRadRng$, namely the free Hopf radical ring associated with a Hopf brace. According to the construction presented in Proposition \ref{free construction for full functors}, if we take a Hopf brace $H$, $FH$ can be described as a coequalizer in $\HBr$. A priori, we should consider all axioms that hold in $\HRadRng$, but since there is a presentation of $\Th_\RadRng$ involving just two additional axioms with respect to those of $\Th_{\SKB}$, we can further simplify the description of the coequalizer.
Denoted by $(\cdot,\bullet,1,S,T)$ the operations in $\Th_\SKB$, we take $t_1, t_2\colon T^2\to T$ and $s_1,s_2\colon T^3\ra T$ defined by:
\begin{align*}
    t_1=&\cdot,\\
    t_2=&\cdot\circ(\pi_2,\pi_1),\\
    s_1=&\bullet\circ(\cdot\times \id),\\
    s_2=&\cdot\circ(\cdot\times \id)\circ(\bullet\times S\times \bullet)\circ(\pi_1,\pi_3,\pi_4,\pi_2,\pi_5)\circ(\id\times\id\times \Delta^{(2)}),
\end{align*}
where $(\pi_2,\pi_1)\colon T^2\to T^2$ is the map induced by the universal property of the product by means of the first and second projection $\pi_1,\pi_2\colon T^2\to T$. Similarly, $(\pi_1,\pi_3,\pi_4,\pi_2,\pi_5)\colon T^5\to T^5$ is the map induced by the universal property of the product $T^5$ by means of the projections $\{\pi_i\colon T^5\to T\}_{i=1,\dots,5}$ on the respective components.
Notice that their images under $H$ applied to elements $a\otimes b\in H^{\otimes 2}$ and $a\otimes b\otimes c\in H^{\otimes 3}$ are just:
\begin{align*}
    t_1^H(a\otimes b)=&a\cdot b,\\
    t_2^H(a\otimes b)=&b\cdot a,\\
    s_1^H(a\otimes b\otimes c)=&(a\cdot b)\bullet c,\\
    s_2^H(a\otimes b\otimes c)=&(a\bullet c_1)\cdot S(c_2)\cdot (b\bullet c_3),
\end{align*}
which are those appearing in Definition \ref{Def HRadRng}. Hence, applying Proposition \ref{free construction for full functors}, we get that $FH$ is the coequalizer of the arrows
$$H^{\otimes 2}\sqcup H^{\otimes 3}\xbigtoto[{[t_1^H,s_1^H]}]{[t_2^H,s_2^H]} H\xrightarrow{\hspace{0.8cm}}FH.$$
Furthermore, the description of coequalizers in $\ref{coequalizer construction}$ provides the explicit construction of this coequalizer: $FH$ is the quotient of $H$ by the $\Th_\SKB$-ideal generated by $I\coloneqq\{a\cdot b-b\cdot a,(a\cdot b)\bullet c -(a\bullet c_1)\cdot S(c_2)\cdot (b\bullet c_3)\mid a,b,c\in H\}$. This $\Th_\SKB$-ideal is the ``linearized'' analogue of the ``radicalator'' defined in \cite{GranLetourmyVendramin} in order to construct the free radical ring generated by a skew brace.

Now, let us show that the algebraic functor, in this case, determines a Birkhoff subvariety, as in the classical case. 

\begin{Def}
    A \textit{Birkhoff subcategory} $\mathsf{D}$ of a category $\mathsf{C}$ is a regular epi-reflective full subcategory that is closed under quotients, namely whenever we have a regular epimorphism $q\colon D\to Q$ and $D$ is in $\mathsf{D}$, $Q$ also lies in $\mathsf{D}$.
\end{Def}

\begin{Prop}\label{surj determines Birkhoff subvariety}
    Let $R\colon \S\ra \Th$ be a surjective map of theories. Then $[\Th,\Coalg]$ is a Birkhoff subcategory of $[\S,\Coalg]$.
\end{Prop}
\begin{proof}
    Since $-\circ R$ is fully faithful and has a left adjoint, $[\Th,\Coalg]$ is a full subcategory of $[\S,\Coalg]$. Let us check that it is closed under quotients. Let $\alpha\colon \mathbf A\circ R \ra \mathbf B$ be a regular epi with $\mathbf A$ a coalgebraic model of $\Th$, and $\mathbf B$ a coalgebraic model of $\S$; in order to show that $\mathbf B$ is also a coalgebraic model for $\Th$, we should check that every axiom $t=s$ for some $s,t\colon T^n\to T$ in  $\Th$ holds in $B$. But the equalities
    $$s^B\circ \alpha_{S^n}=\alpha_{S^1}\circ s^A=\alpha_{S^1}
    \circ t^A=t^B\circ \alpha_{S^n}$$
    \noindent imply that $s^B=t^B$, since $\alpha$ is componentwise surjective (it follows from the description of coequalizers provided in \ref{coequalizer construction}).
    \noindent Finally, we have to show that the category $[\Th,\Coalg]$ is regular-epi reflective, but for any $\mathbf C\in[\S,\Coalg]$, the unit of the adjunction
    $\eta_\mathbf{C}\colon \mathbf C\to L\mathbf C$
    is exactly the coequalizer $q$ described in Proposition \ref{free construction for full functors}, and so it is a regular epi.
\end{proof}

As a consequence of this proposition we obtain that, in the diagram \ref{Chain of Hopf inclusions}, the functors 
$$\HRadRng\longrightarrow\HBr\quad\quad \HBr\longrightarrow\HDiGrp$$ 
are inclusions of Birkhoff subcategories, analogously to what happens for the classical algebraic functors $\RadRng\to\SKB$ and $\SKB\to\HDiGrp$.

\end{section}

\begin{section}{$\Omega$-Hopf algebras}\label{omega groups}

In light of the importance of the category $\Hopfc$, we want to study pointed categories of coalgebraic models that come with a forgetful functor to $\Hopfc$. In particular, we are going to show that these categories inherit some good properties---semi-abelianness above all---from cocommutative Hopf algebras.    
\\

Before proving the main result of this section, we provide a slightly different way of describing these categories, that will also somehow explain their relevance. In fact, looking at Table \ref{Table of examples}, one can notice that many examples fit in this framework: in particular, this is the case for the categories  $\Hopfc$, $\Hopfcc$, $\HBr$, $\HDiGrp$ and $\HRadRng$. As a matter of fact, all these categories are examples of categories of ``$\Omega$-Hopf algebras'' (Definition \ref{def omega hopf algebras}), i.e. coalgebraic models of an algebraic theory of $\Omega$-groups. 
It turns out that these two classes of categories coincide, namely a pointed category of coalgebraic models of an algebraic theory has a forgetful functor to $\Hopfc$ exactly when it is a category of $\Omega$-Hopf algebras. 





In order to prove the last sentence, let us spell out the definition of an $\Omega$-Hopf algebra. 
\\

We observe that being a Lawvere theory $\Th$ of a variety of $\Omega$-groups means that $\Th$ has operations $\cdot\colon T^2\to T$, $1\colon T^0\to T$, $ S\colon T\to T$ satisfying the group axioms, there are no other maps from $T^0$ to $T$ and, for any map $f\colon T^n\to T$ in $\Th$ the axiom $f\circ \Delta^{(n-1)}\circ 1=1$ holds.

    Hence, explicitly, an $\Omega$-Hopf algebra is a cocommutative Hopf algebra $(H,\cdot^H,1^H,\Delta_H,\epsilon_H,S^H)$ endowed with a set of maps $F$ satisfying the following properties:
    \begin{itemize}
        \item an element $t\in F$ is a map $t\colon H^{\otimes n}\to H$,
        \item every map $t\colon H^{\otimes n}\to H$ in $F$ is a homomorphism of coalgebras,
        \item for every $n$-ary operation $t\in F$ the axiom $t(1^H\otimes \cdots \otimes 1^H)=1^H$ holds.
    \end{itemize}

    We can now prove the following proposition.

\begin{Prop}\label{Characterization of Omega-groups}
    A category of coalgebraic models of a Lawvere theory $\Th$ is a category of $\Omega$-Hopf algebras if and only if it is pointed and there exists an algebraic functor $[\Th,\Coalg]\to \Hopfc$.
\end{Prop}
\begin{proof}
  If $[\Th,\Coalg]$ is a category of $\Omega$-Hopf algebras, it trivially has a forgetful functor to $\Hopfc$. Hence, we just have to check that it is pointed. Let us call $u\colon T^0\to T$ the neutral element of $\Th$. For any $\Omega$-Hopf algebra $H$, $u$ induces a natural transformation $u^H\colon \K\Rightarrow H$. Indeed, the naturality follows from the fact that, for any $f\colon T^n\to T$,
  $$f^H\circ (u^H)^{\otimes n}=f^H\circ\Delta_H^{(n-1)}\circ u^H=u^H.$$
  Moreover, the natural transformation $u^H$ is the unique possible one: given another $\alpha\colon \K\Rightarrow H$, the naturality with respect to $u$ ensures that $u^H=\alpha$.

  For the other implication, we start from a pointed category $[\Th,\Coalg]$ with a forgetful functor to $\Hopfc$; $\Th$ is endowed with group structure coming from the algebraic functor. Furthermore, the unique morphism from $\K$ to a model $H$ must coincide with the unit, since it has to be natural with respect to the unit operation. Hence, the unit is a natural transformation, and the naturality condition translates into the $\Omega$-Hopf algebra requirement on the operations, as in the argument above. 
\end{proof}

\begin{subsection}{Semi-abelianness of $\Omega$-Hopf algebras}

Let us introduce some notation and recall some results about the category $\Hopfc$. Given a map $f\colon A\ra B$ in $\Hopfc$, we denote by $\Hker{f}$ its kernel in $\Hopfc$ and by $\ker {f}$ the usual kernel as a map of vector spaces. It is shown in \cite{GranSterckVercruysse} that the category $\Hopfc$ has a regular epi-mono factorization, where the image of the factorization is given by 
$$\frac{A}{A(\Hker{f})^+}=\frac{A}{\ker{f}}\cong f[A].$$
This follows from the equality $A(\Hker{f})^+=\ker{f}$ that is a consequence of the Newman's correspondence \cite{Newman}.

\begin{Prop}\label{factorization}
    Let $[\Th,\Coalg]$ be a category of $\Omega$-Hopf algebras. Then any morphism in $[\Th,\Coalg]$ factors through a regular epimorphism followed by a monomorphism.
\end{Prop}
\begin{proof}
    
    Take a map $f\colon A\ra B$ in $[\Th,\Coalg]$ and let $H$ be the kernel of $f$ in $[\Th,\Coalg]$. Since the forgetful functor to $\Hopfc$ is a right adjoint, $H$ is computed as in $\Hopfc$, i.e.
    $$H=\Hker{f}=\{a\in A\mid f(a_1)\otimes a_2=1_B\otimes a\}.$$
    The map $f$ factors through the cokernel of $H$ that is, by Remark \ref{Cokernel}, $\frac{A}{\langle \Hker{f}^+\rangle_{\Th}}$. Let us prove that $\langle \Hker{f}^+\rangle_{\Th}=\ker{f}$. Since $\ker f$ is a $\Th$-ideal (see Remark \ref{ker ideal}) and $\Hker{f}^+\subseteq \ker f$, we get the inclusion $\langle \Hker{f}^+\rangle_\Th\subseteq \ker f$. On the other hand, it is clear that $\langle \Hker f^+\rangle_\Th\supseteq A \Hker f^+=\ker f$. Now, as in the category $\Hopfc$, we can factor $f$ through the cokernel of its kernel
    as in the diagram \begin{figure}[H]
        \centering
    \begin{tikzpicture}
        \node(1) at (-1,0) {$A$};
        \node(2) at (3,0) {$B$};
        \node(3) at (1,-1) {$\frac{A}{\langle \Hker f^+\rangle_\Th}$};

        \draw[->] (1) to node[above]{\scriptsize{$f$}} (2);
        \draw[->] (1) to[in=170,out=-50] node[below] {\scriptsize{$e$}} (3);
        \draw[->] (3) to[in=230,out=10] node[below] {\scriptsize{$m$}} (2);
    \end{tikzpicture}    
    \end{figure}
    \noindent where the map $e$ is trivially a regular epi and the map $m$ is a monomorphism (it is, in fact, injective since $\frac{A}{\langle \Hker f^+\rangle_\Th}=\frac{A}{\ker f}=f[A]$).
    
\end{proof}

\begin{Prop}\label{reg epi=surj}
    Let $[\Th,\Coalg]$ be a category of $\Omega$-Hopf algebras. In $[\Th,\Coalg]$ normal epimorphisms coincide with regular epimorphism and with surjective maps.
\end{Prop}
\begin{proof}
    Using the description of coequalizers provided in \ref{coequalizer construction}, we know that every regular epimorphism (and so every normal epi) is surjective, hence
    $$\{\text{normal epi}\}\subseteq\{\text{regular epi}\}\subseteq\{\text{surjective maps}\}.$$

    On the other hand, given a surjective map $f\colon A\to B$ in $[\Th,\Coalg]$, its regular epi-mono factorization coincides with the usual factorization in $\K$-vector spaces. Hence, the monomorphic part of the factorization is an isomorphism, i.e. $f$ coincide---up to iso--- with the regular epi part of its factorization. Furthermore, the description of the factorization provided in Proposition \ref{factorization} shows that the regular epi part of the factorization is, in fact, a normal epi. Therefore, every surjective map is a normal epi.
\end{proof}

\begin{Prop}
    Let $[\Th,\Coalg]$ be a category of $\Omega$-Hopf algebras and let $A$ be a coalgebraic model of $\Th$. There is a one-to-one correspondence between 
    $$I\coloneqq\{B\subset A\mid B \text{ is a sub-}\Th\text{-coalgebra of }A \text{ s.t. } \langle B^+\rangle_\Th= AB^+\}$$
    and
    $$J\coloneqq\{\text{ normal monomorphism }B\hookrightarrow A\}.$$
\end{Prop}
\begin{proof}
    Let $B\in I$, we can consider in $[\Th,\Coalg]$ the map $\pi\colon A\ra \faktor{A}{\langle B^+\rangle_\Th} $. Since $\langle B^+\rangle_\Th= AB^+$, due to the Newman correspondence, we know that $B$ is the kernel of $\pi$ in $\Hopfc$ and then also in $[\Th,\Coalg]$. Conversely, when $B=\Hker f $ is the kernel $B\hookrightarrow A $ of a morphism $ A\xrightarrow{f} A'$, then, by using the same argument of Proposition \ref{factorization}, we get that $\langle \Hker f ^+\rangle_\Th=\ker f=A\Hker f^+$.
\end{proof}

\begin{Prop}\label{image of mono is mono}
    Let $[\Th,\Coalg]$ be a category of $\Omega$-Hopf algebras. Then, the direct image of a kernel in $[\Th,\Coalg]$ under a regular epimorphism is a kernel in $[\Th,\Coalg]$.
\end{Prop}

\begin{proof}
    Let $D\hookrightarrow A$ be a normal subobject of $A$ and let $\rho\colon A\to B$ be a surjective map. Thanks to the previous proposition, it is enough to check that $\langle \rho[D]^+\rangle_\Th\subseteq B\, \rho[D]^+$. 
    
    Let us call $I\coloneqq \rho[D]^+$ and $J=D^+$. Adopting the same notation of Proposition \ref{T ideal generated by a coideal}, we prove by induction that $I_j\subseteq \rho[J_j]$ for any $j\in\N$. For $j=0$, it is straightforward since $$I_0=I=\rho[D]^+=\rho[D^+]=\rho[J_0].$$
    Now, we assume that $I_j\subseteq\rho[J_j]$ and we prove that $I_{j+1}\subseteq\rho[J_{j+1}]$. It is enough to check the thesis for the generators of $I_{j+1}$ as linear subspace. If $x\in I_j$, by induction assumption, we have that $x\in\rho[J_j]\subseteq \rho[J_{j+1}]$. In case $x=t^B(x_0\otimes\cdots\otimes x_{n-1})$ for $n\in \N$, $t$ an $n$-ary term and $x_0,\cdots,x_{n-1}\in B$ such that $\exists i\in\{0,\cdots,n-1\}$ for which $x_i\in I_j$, we can assume, without loss of generality, that $i=0$. Since $x_0\in I_j\subseteq\rho[J_j]$, there is an element $d\in J_j$ such that $x_0=\rho(d)$. Furthermore, since $\rho$ is surjective, there are elements $a_1,\dots,a_{n-1}\in A$ such that $\rho(a_i)=x_i$ for every $i=1,\dots,n-1$. Hence, we can compute the element $x$ as follows:
    $$x=t^B(x_0\otimes \cdots\otimes x_{n-1})=t^B(\rho(d)\otimes \rho(a_1)\otimes\cdots\otimes \rho(a_{n-1}))=\rho(t^A(d\otimes a_1\otimes\cdots\otimes a_{n-1})).$$
    But we also have that $t^A(d\otimes a_1\otimes a_{n-1})\in J_{j+1}$, thus we get that $x\in \rho[J_{j+1}]$.

    Therefore, we have obtained that $I_j\subseteq\rho[J_j]$ for any natural number $j$. This implies that $$\langle \rho[D]^+\rangle_\Th=\langle I\rangle_\Th\subseteq\rho[\langle J\rangle_\Th]=\rho[\langle D^+\rangle_\Th],$$
    but, since $D$ is a kernel, we also have that $\langle D^+\rangle_\Th=AD^+$. So, we finally obtain that
    $$\langle \rho[D]^+\rangle_\Th\subseteq \rho[AD^+]\subseteq B\rho[D]^+.$$
\end{proof}

\begin{Prop}\label{category with forgetful homological}
    Let $[\Th,\Coalg]$ be a category of $\Omega$-Hopf algebras. Then $[\Th,\Coalg]$ is a protomodular category.
\end{Prop}
\begin{proof}
    Let us recall that Proposition \ref{Protomodularity mal'tsev condition} characterizes protomodularity as a Mal'tsev condition, i.e. $[\Th,\Coalg]$ is protomodular if there  are suitable terms satisfying certain equations. The category $[\Th,\Coalg]$ comes with a forgetful functor to $\Hopfc$, namely there is a functor of Lawvere theories $F\colon\Grp\to \Th$. But there are terms in $\Grp$ which ensure the protomodularity of $\Hopfc$, hence their image under $F$ satisfy the same commutative diagrams in $\Th$ and guarantee the protomodularity of $[\Th,\Coalg]$.
\end{proof}

\begin{Prop}\label{Prop:regularity}
    Let $[\Th,\Coalg]$ be a category of $\Omega$-Hopf algebras. Then $[\Th,\Coalg]$ is regular.
\end{Prop}
\begin{proof}
    We have already proved the existence of a regular epi-mono factorization in Proposition \ref{factorization}. To prove that regular epimorphisms are stable under pullback we can use that the forgetful functor $U$ to the category $\Hopfc$ preserves and reflects regular epi (from Proposition \ref{reg epi=surj} we know that regular epimorphisms coincide with surjective maps in both categories). We take a regular epi $f\colon \mathbf A\to \mathbf B$, a map $p\colon\mathbf C\to \mathbf B$ in $[\Th,\Coalg]$, and the pullback $f^*$ of $f$ along $p$. Since $U$ preserves regular epimorphisms and pullbacks, $Uf$ is a regular epi and $Uf^*$ is its pullback along $Up$. But $\Hopfc$ is a regular category, so regular epimorphisms in $\Hopfc$ are stable under pullback, hence $Uf^*$ is a regular epi in $\Hopfc$. Using the fact that $U$ reflects regular epimorphisms, we conclude that $f^*$ is a regular epi in $[\Th,\Coalg]$. 
\end{proof}

We can now put together the previous propositions and obtain the following result.

\begin{Thm}\label{semiabelianess of omega hopf algebras}
    Let $[\Th,\Coalg]$ be a category of $\Omega$-Hopf algebras. Then $[\Th,\Coalg]$ is semi-abelian.   
\end{Thm}
\begin{proof}
    In order to show that $[\Th,\Coalg]$ is semi-abelian, we can apply 3.7 of \cite{JanelidzeMarkiTholen}. Indeed, the category is pointed, has all finite limits and colimits, and every morphism $f$ with $\Hker f=\K$ is a mono---since, if $\Hker f=\K$, the regular epi part of the factorization of $f$ is trivial (see Proposition \ref{factorization}), namely $f$ is a mono. Furthermore, the Split Short Five Lemma holds (see Proposition \ref{category with forgetful homological}), normal epimorphisms are stable under pullbacks (it follows from Proposition \ref{Prop:regularity} and Proposition \ref{reg epi=surj}), and we proved in Proposition \ref{image of mono is mono} that images of normal monomorphisms under normal epimorphisms are normal monorphisms.
\end{proof}

Theorem \ref{semiabelianess of omega hopf algebras} immediately applies to all the above-mentioned examples of $\Omega$-Hopf algebras. This leads to a different proof of the result about semi-abelianness of $\HBr$ obtained in \cite{GranSciandra} and to the following corollary.

\begin{Cor}\label{semiabelianess of Hradrng and Hdigrp}
The categories $\HRadRng$ and $\HDiGrp$ of Hopf radical rings and of Hopf digroups are semi-abelian.
\end{Cor}

Theorem \ref{semiabelianess of omega hopf algebras} paves the way for the study of non-abelian homology for categories of $\Omega$-Hopf algebras, along the lines of \cite{gransciandrahopfformulaecocommutativehopf} in which the case of cocommutative Hopf algebras is analyzed. It would also be interesting to investigate models of Lawvere theories in other monoidal categories. For instance, one may consider the category of graded coalgebras: in this setting, models of the algebraic theory of groups correspond to cocommutative color Hopf algebras, whose category has been shown to be semi-abelian under mild assumptions \cite{SciandracolorHopf}. Another possible direction concerns the study of models of algebraic theories in the category of coalgebras over a von Neumann regular ring rather than over a field \cite{Porst2011}.

\end{subsection}
\end{section}

\section*{Acknowledgement}{
I would like to express my gratitude to my supervisor, Prof. Marino Gran, for his invaluable guidance, insightful suggestions, and constant support throughout this work. I am also grateful to Andrea Sciandra for his precious feedback on an earlier version of the paper and to Prof. Matìas Menni for providing helpful comments.
Finally, I would like to thank the anonymous referee for his careful reading of the manuscript and for his constructive comments and suggestions, which have significantly improved the quality and clarity of this paper.

}

\section*{Declarations}{
\textbf{Funding:} The author's research is funded by a FRIA doctoral grant from the \textit{Fonds de la Recherche Scientifique --FNRS}.\\
\\
\textbf{Conflict of interest:} The author declares no Conflict of interest.\\
\\
\textbf{Data availability:} No datasets were generated or analysed during the current study.\\
\\
\textbf{Author contributions:} The preprint was prepared entirely by the only author which is also the corresponding
author.\\
\\
\textbf{Ethics declaration:} Not applicable.

}

\printbibliography[
heading=bibintoc,
title={Bibliography}
]

\end{document}